\documentclass{CRPITStyle}
\usepackage[authoryear]{natbib}
\renewcommand{\cite}{\citep}
\usepackage{palatino}
\usepackage{amsthm}
\usepackage{amsmath}
\usepackage{amssymb}
\usepackage{enumerate}
\usepackage{varioref}
\usepackage{charter,eulervm}
\usepackage{float}
\usepackage{eepic}

\newcommand{\es}{\varnothing}
\newtheorem{theorem}{Theorem}

\newtheorem{lemma}{Lemma}

\newtheorem{definition}{Definition}
\newtheorem{remark}{Remark}

\newcommand\conferencenameandplace{17th Computing: The Australasian Theory Symposium (CATS 2011), Perth, Australia,
January 2011}
\newcommand\volumenumber{119}
\newcommand\conferenceyear{2011}
\newcommand\editorname{Alex Potanin and Taso Viglas}

\title{{\sc Black-and-White Threshold Graphs}}

\author{
 Ling-Ju~Hung$^1$
\and
 Ton~Kloks%
\and
 Fernando~S.~Villaamil$^2$
}

\affiliation{$^1$
 Department of Computer Science and Information Engineering\\
 National Chung Cheng University,
 Chia-Yi 621, Taiwan\\
Email:~{\tt hunglc@cs.ccu.edu.tw}\\
 $^2$Department of Computer Science\\
 RWTH Aachen University, Aachen 52056, Germany\\
 Email:~{\tt fernando.sanchez@rwth-aachen.de}
}

\begin{document}

\maketitle

\toappear{This research was supported by the National Science
Council of Taiwan, under grants NSC~98--2218--E--194--004 and
NSC~98--2811--E--194--006. Ton Kloks is supported by the National
Science Council of Taiwan, under grants NSC~99--2218--E--007--016
and NSC~99--2811--E--007--044.}
 \toappearstandard
\begin{abstract}
Let $k$ be a natural number. We introduce {\em $k$-threshold
graphs}. We show that there exists an $O(n^3)$ algorithm for the
recognition of $k$-threshold graphs for each natural number $k$.
$k$-Threshold graphs are characterized by a finite collection of
forbidden induced subgraphs. For the case $k=2$ we characterize the
{\em partitioned 2-threshold graphs} by forbidden induced subgraphs.
We introduce {\em restricted}, and {\em special 2-threshold graphs}.
We characterize both classes by forbidden induced subgraphs. The
restricted 2-threshold graphs coincide with the {\em switching class
of threshold graphs}. This provides a decomposition theorem for the
switching class of threshold graphs.
\end{abstract}

\section{Introduction}

A graph is a pair $G=(V,E)$ where $V$ is a finite, nonempty set and
where $E$ is a set of two-element subsets of $V$. We call the
elements of $V$ the vertices or points of the graph. We denote the
elements of $E$ as $(x,y)$ where $x$ and $y$ are vertices. We call
the elements of $E$ the edges of the graph. For two sets $A$ and
$B$, we use $A+B$ and $A-B$ to denote $A\cup B$ and $A\setminus B$
respectively. For a set $A$ and an element $x$, denote $A\cup \{x\}$
and $A\setminus \{x\}$ as $A+x$ and $A-x$ respectively. If $e=(x,y)$
is an edge of a graph then we call $x$ and $y$ the endpoints of $e$
and we say that $x$ and $y$ are adjacent. The {\em open
neighborhood} of a vertex $x$ is the set of vertices $y$ such that
$(x,y) \in E$. We denote open neighborhood by $N(x)$. Define the
{\em closed neighborhood} of $x$ by $N(x)+x$. We use $N[x]$ to
denote the closed neighborhood of $x$. Define $\Delta(x)=V-N[x]$. We
call $\Delta(x)$ the {\em anti-neighborhood} of $x$. The degree of a
vertex $x$ is the cardinality of $N(x)$. Let $W \subseteq V$ and let
$W \neq \es$. The graph $G[W]$ induced by $W$ has $W$ as its set of
vertices and it has those edges of $E$ that have both endpoints in
$W$. If $W \subset V$, $W \neq V$, then we write $G-W$ for the graph
induced by $V \setminus W$. In case $W$ consists of a single vertex
$x$ we write $G-x$ instead of $G-\{x\}$. We usually denote the
number of vertices of a graph by $n$.

A path is a graph of which the vertices can be linearly
ordered such that the pairs of consecutive vertices form the
set of edges of the graph. We call the first and last vertex in this
ordering the terminal vertices of the path.
We denote a path with $n$ vertices by $P_n$. To denote a specific
ordering of the vertices we use the notation $[x_1,\ldots,x_n]$.
Let $G$ be a graph. Two vertices $x$ and $y$ of $G$ are connected by a
path if $G$ has an induced subgraph which is a path with $x$
and $y$ as terminals. Being connected by a path is an equivalence relation
on the set of vertices of the graph.
The equivalence classes are called the
components of $G$.
A cycle consists of a path with at least three vertices with
one additional edge that connects the two terminals of the path.
We denote a cycle with $n$ vertices by $C_n$.
The length of a path or a cycle is its number of edges.

\section{$\mathbf{k}$-Threshold graphs}

Threshold graphs were introduced in~\cite{kn:chvatal}
using a concept called `threshold dimension.'
There is a lot of information about threshold graphs in the
book~\cite{kn:mahadev}, and
there are chapters on threshold graphs
in the book~\cite{kn:golumbic2} and in the
survey~\cite{kn:brandstadt}.

There are many ways to define threshold graphs.
We choose the
following way~\cite{kn:chvatal,kn:brandstadt}.
An isolated vertex in a graph $G$ is a vertex without neighbors.
A universal vertex is a vertex that is adjacent to all other
vertices. A pendant vertex is a vertex with exactly one neighbor.

\begin{definition}
\label{def threshold}
A graph $G=(V,E)$ is a {\em threshold graph\/} if
every induced subgraph has an isolated vertex or a
universal vertex.
\end{definition}

A graph $G$ is a threshold graph if and only if
$G$ has no induced
$P_4$, $C_4$, nor $2K_2$~\cite{kn:chvatal}.

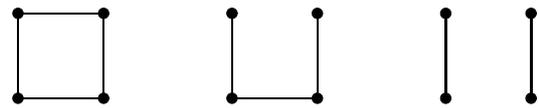
\begin{figure}
\setlength{\unitlength}{3.2pt}
\begin{center}
\begin{picture}(60,10)
\put(0,0){\circle*{1.2}}
\put(0,10){\circle*{1.2}}
\put(10,0){\circle*{1.2}}
\put(10,10){\circle*{1.2}}
\put(0,0){\line(1,0){10}}
\put(0,0){\line(0,1){10}}
\put(10,0){\line(0,1){10}}
\put(10,10){\line(-1,0){10}}

\put(25,0){\circle*{1.2}}
\put(25,10){\circle*{1.2}}
\put(35,0){\circle*{1.2}}
\put(35,10){\circle*{1.2}}
\put(25,0){\line(1,0){10}}
\put(25,0){\line(0,1){10}}
\put(35,0){\line(0,1){10}}

\put(50,0){\circle*{1.2}}
\put(50,10){\circle*{1.2}}
\put(60,0){\circle*{1.2}}
\put(60,10){\circle*{1.2}}
\put(50,0){\line(0,1){10}}
\put(60,0){\line(0,1){10}}
\end{picture}
\end{center}
\caption{$C_4$, $P_4$ and $2K_2$}
\end{figure}

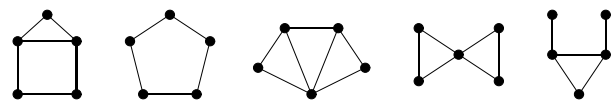
\begin{figure}[h]
\setlength{\unitlength}{1pt}
\begin{center}
\begin{picture}(220,30)(0,0)
\put(0,0){\circle*{3.4}} \put(0,20){\circle*{3.4}}
\put(11,30){\circle*{3.4}} \put(22,20){\circle*{3.4}}
\put(22,0){\circle*{3.4}} \put(0,0){\line(0,1){20}}
\put(0,0){\line(1,0){22}} \put(22,0){\line(0,1){20}}
\put(0,20){\line(1,0){22}} \put(0,20){\line(11,10){11}}
\put(0,20){\line(1,0){20}} \put(11,30){\line(11,-10){11}}

\put(42,20){\circle*{3.4}} \put(47,0){\circle*{3.4}}
\put(67,0){\circle*{3.4}} \put(72,20){\circle*{3.4}}
\put(57,30){\circle*{3.4}} \put(47,0){\line(-1,4){5}}
\put(47,0){\line(1,0){20}} \put(67,0){\line(1,4){5}}
\put(57,30){\line(3,-2){16}} \put(57,30){\line(-3,-2){16}}

\put(90,10){\circle*{3.4}} \put(110,0){\circle*{3.4}}
\put(100,25){\circle*{3.4}} \put(120,25){\circle*{3.4}}
\put(130,10){\circle*{3.4}} \put(90,10){\line(2,-1){20}}
\put(90,10){\line(2,3){10}} \put(100,25){\line(1,0){20}}
\put(100,25){\line(2,-5){10}} \put(120,25){\line(2,-3){10}}
\put(110,0){\line(2,5){10}} \put(110,0){\line(2,1){20}}

\put(150,5){\circle*{3.4}} \put(150,25){\circle*{3.4}}
\put(165,15){\circle*{3.4}} \put(180,5){\circle*{3.4}}
\put(180,25){\circle*{3.4}} \put(150,5){\line(0,1){20}}
\put(180,5){\line(0,1){20}} \put(150,5){\line(3,2){30}}
\put(150,25){\line(3,-2){30}}

\put(200,15){\circle*{3.4}} \put(220,15){\circle*{3.4}}
\put(210,0){\circle*{3.4}} \put(200,30){\circle*{3.4}}
\put(220,30){\circle*{3.4}} \put(200,30){\line(0,-1){15}}
\put(220,30){\line(0,-1){15}} \put(200,15){\line(1,0){20}}
\put(210,0){\line(2,3){10}} \put(210,0){\line(-2,3){10}}
\end{picture}
\end{center}
\caption{A house, a $C_5$, a gem, a butterfly, and a bull}
\label{fig:house_C5_gem}
\end{figure}

A decomposition tree for a graph $G=(V,E)$ is a pair $(T,f)$ where
$T$ is a rooted binary tree and where $f$ is a bijection from the
vertices of $G$ to the leaves of~$T$.

Let $G$ be a threshold graph. We can construct a decomposition tree
for $G$ as follows. For each internal node, including the root, the
right subtree consists of a single leaf. Each internal node is
labeled with a $\oplus$-operator or a $\otimes$-operator. Consider
an internal node and let $W$ be the set of vertices that are mapped
to leaves in the left subtree. The $\oplus$-operator adds the vertex
that is mapped to the leaf in the right subtree as an isolated
vertex to the graph induced by $W$. The $\otimes$-operator adds the
vertex that is mapped to the leaf in the right subtree as a
universal vertex to the graph induced by $W$.

\bigskip
Let $C=\{1,\ldots,k\}$ be a set of $k$ colors. A coloring of
a graph $G=(V,E)$ is a map $c:V \rightarrow C$.

\begin{definition}
A graph $G=(V,E)$ is a {\em $k$-threshold graph\/} if there is a
coloring of $G$ with $k$ colors
such that for every nonempty set
$W \subseteq V$ the graph $G[W]$
has a vertex which is either isolated, or a vertex $x$ which is
adjacent to exactly all vertices of color $i$ in $W-x$ for some
color $i$.
\end{definition}

A $k$-threshold graph $G$ has a decomposition tree $(T,f)$ where $T$
is a rooted binary tree and $f$ is a bijection from the leaves of
$T$ to the vertices of $G$. Each leaf has one color chosen from a
set of $k$ colors. For each internal node the right subtree consists
of a single leaf. Each internal node including the root is labeled
with a $\oplus$-operator or a $\otimes_i$-operator. The
$\oplus$-operator adds the vertex that is mapped to the leaf in the
right subtree as an isolated vertex. The $\otimes_i$-operator makes
the vertex of the right leaf adjacent exactly to the vertices of
color $i$ in the left subtree.

\bigskip
Consider the class of $k$-threshold graphs. Notice that the
class is hereditary; that is, if $G$ is a $k$-threshold graph
then so is every induced subgraph of $G$. Any hereditary class of graphs
is characterized by a set of forbidden induced subgraphs.
These are the minimal elements in the induced subgraph relation
that are not in the class.
For example, a graph is a threshold graph if and only if it has
no induced $2K_2$, $C_4$ or $P_4$.
In the following theorem we show that this characterization is
finite for $k$-threshold graphs for any fixed number $k$.

\begin{theorem}
\label{kruskal}
Let $k$ be a natural number.
The class of $k$-threshold graphs
is characterized by a finite collection of forbidden
induced subgraphs.
\end{theorem}
\begin{proof}
To prove this theorem we use the technique
introduced by Pouzet~\cite{kn:pouzet}.

Let $T_1,T_2,\ldots$ be a collection of
rooted binary trees with points labeled from some finite set.
We write $T_i \prec T_j$ if there exists an injective map
$h$ from the points of $T_i$ to the points of $T_j$
such that
\begin{enumerate}[\rm 1.]
\item the label of a point $a$ in $T_i$ is equal to the
label of the point $h(a)$ in $T_j$, and
\item for every pair of points $a$ and $b$ in $T_i$, their lowest
common ancestor is mapped to the lowest common ancestor of $h(a)$
and $h(b)$ in $T_j$, and

\item  if $a$ and $b$ are points of $T_i$ with lowest common
ancestor $c$ such that $a$ is in the left subtree of $c$ and $b$ is
in the right subtree of $c$ then $h(a)$ is in the left subtree of
$h(c)$ and $h(b)$ is in the right subtree of $h(c)$ in $T_j$.
\end{enumerate}

Let $T_1,T_2,\ldots$ be an infinite sequence of rooted binary trees
with points labeled from some finite set. Kruskal's
theorem~\cite{kn:kruskal} states that there exist integers $i < j$
such that $T_i \prec T_j$.

Assume that the class of
$k$-threshold graphs has an infinite collection of
minimal forbidden induced subgraphs, say $G_1,G_2,\ldots$.
In each $G_i$ single out one vertex $r_i$ and let
$G_i^{\prime}=G_i-r_i$. Then $G_i^{\prime}$ is a $k$-threshold graph.
For each $i$ consider a
decomposition tree $(T_i,f_i)$ for $G_i^{\prime}$
as described above. We add one more
label to each leaf of $T_i$. This label is $1$ if the vertex that is
mapped to that leaf is adjacent to $r_i$ and it is $0$ otherwise.

When we apply Kruskal's theorem to the labeled binary trees $T_i$
that represent the graphs $G_i^{\prime}$ we may conclude that there
exist $i < j$ such that $G_i^{\prime}$ is an induced subgraph of
$G_j^{\prime}$. By virtue of the extra label we have also that $G_i$
is an induced subgraph of $G_j$. This is a contradiction because we
assume that the graphs $G_i$ are minimal forbidden induced
subgraphs. This proves the theorem.
\end{proof}

Higman's lemma~\cite{kn:higman}
preludes Kruskal's theorem. It deals with
finite sequences over a finite alphabet instead of trees.
Instead of Kruskal's theorem we could have used Higman's lemma
to prove Theorem~\ref{kruskal}.

\bigskip
Let $\mathcal{O}_k$ be the set of forbidden induced subgraphs
for $k$-threshold graphs. The set $\mathcal{O}_k$ is called the
obstruction set for $k$-threshold graphs.

\bigskip
The most natural way to express and classify graph-theoretic problems
is by means of logic. In monadic second order logic a (finite) sentence is a
formula that uses quantifiers $\forall$ and $\exists$. The quantification
is over vertices, edges, and subsets of vertices and edges. Relational
symbols are $\neg$, $\in$, $=$, $\mbox{and}$, $\mbox{or}$, $\subseteq$,
$\cup$,
$\cap$, and the logical implication
$\Rightarrow$. Some of these are superfluous.
Although the minimization or maximization of the
cardinality of a subset is not part of the logic, one usually includes
them.

A restricted form of this logic is where one
does not allow quantification over subsets of edges.
The $C_2MS$-logic is such a restricted monadic second-order logic
where one can furthermore use a test whether the cardinality of a subset
is even or odd.

Courcelle proved that problems that can be expressed in $C_2MS$-logic
can be solved in $O(n^3)$ time for graphs of bounded rankwidth
(see~\cite{kn:courcelle3,kn:hlineny}).
Consider a decomposition tree $(T,f)$
for a graph $G=(V,E)$. Let $\ell$ be a line in this tree. Let $V_{\ell}$
be the set of vertices that are mapped to the leaves in the subtree
rooted at $\ell$. The cutmatrix of $\ell$ is the submatrix of the
$0/1$-adjacency matrix of $G$ that has rows indexed by the vertices of
$V_{\ell}$ and that has columns indexed by the vertices of $V-V_{\ell}$.
A graph has rankwidth $k$ if every cutmatrix has rank over $GF[2]$
at most $k$. For each natural number $k$ there exists an $O(n^3)$
algorithm that constructs a decomposition tree for a graph
$G$ of rankwidth $k$ if it exists~\cite{kn:hlineny2}.

\bigskip
Let $(T,f)$ be a decomposition tree for a $k$-threshold graph. Then
every cutmatrix is of the following shape:
\[\begin{pmatrix} I & 0 \\ 0&0 \end{pmatrix}\]
where $I$ is a submatrix of the $k \times k$ identity matrix.
By `having this shape' we mean that we left out multiple
copies of the same row or column.
It follows that $k$-threshold graphs have rankwidth $k$.

\begin{theorem}
Let $k$ be a natural number.
There exists an $O(n^3)$ algorithm
which checks whether a graph $G$ with $n$
vertices is a $k$-threshold graph.
\end{theorem}
\begin{proof}
$k$-Threshold graphs have rankwidth $k$.
$C_2MS$-Problems can be solved in $O(n^3)$
time for graphs of bounded rankwidth%
~\cite{kn:courcelle3,kn:hlineny2,kn:oum4}. The definition
of $k$-threshold graphs can be formulated in this logic.

An other proof goes as follows. Let $\mathcal{O}_k$ be the
obstruction set for $k$-threshold graphs. We can formulate the
existence of a graph in $\mathcal{O}_k$ as an induced subgraph in
$C_2MS$-logic. Note that this proof is non-constructive; Kruskal's
theorem does not provide the forbidden induced subgraphs.
\end{proof}

\section{Special 2-threshold graphs}

There are many drawbacks to the solution given in the previous
section. The algorithm that constructs a decomposition tree of
bounded rankwidth is far from easy. By the way, it is described in
terms of matroids instead of graphs. Furthermore, the constants that
are involved quickly grow out of space when $k$ increases.

\bigskip
Notice that threshold graphs can be recognized by sorting the
vertices according to increasing degrees. Either the vertex with
highest degree is a universal vertex or the vertex with lowest degree
is an isolated vertex. Delete such a vertex which is either
isolated or universal and repeat the process. The graph is a
threshold graph if and only if this process ends with a single
vertex.

\bigskip
In the remainder of this paper we restrict our attention to the case
where $k=2$. We consider colorings of vertices with two colors, say
black and white. If $G$ is a 2-threshold graph with a
black-and-white coloring of the vertices then the black vertices
induce a threshold graph and the white vertices induce a threshold
graph.

\begin{definition}
A 2-threshold graph is {\em special\/} if it has a black-and-white
coloring and a decomposition tree with only $\oplus$- and
$\otimes_w$-operators.
\end{definition}

A clique in a graph is a nonempty subset of
vertices such that every pair in it is adjacent.
An independent set in a graph is a nonempty subset of
vertices with no edges between them.


A graph is bipartite if there is a partition of the vertices into
two independent sets. One part of the partition may be empty. A
bipartite graph is complete bipartite if any two vertices in
different parts of the partition are adjacent.

A threshold order of a graph is a linear ordering
$[v_1,\ldots,v_n]$ of the vertices such that for every
pair $v_i$ and $v_j$ with $i < j$
\begin{eqnarray*}
N(v_i) \subseteq N(v_j) && \mbox{if $v_i$ and $v_j$ are nonadjacent}\\
\mbox{and}\quad N[v_i] \subseteq N[v_j] &&\mbox{if $v_i$ and $v_j$ are
adjacent.}
\end{eqnarray*}
A graph is a threshold graph if and only if it has a
threshold order~\cite{kn:mahadev}. This implies that the vertices
of a threshold graph can be partitioned into a clique (the higher
degree vertices)
and an independent set (the lower degree vertices).
Note however, that this partition is not exactly unique.

\begin{figure}
\setlength{\unitlength}{1.8pt}
\begin{center}
\begin{picture}(110,20)(0,35)

\put(2.5,45){\circle*{2}} \put(7.5,45){\circle*{2}}
\put(12.5,45){\circle*{2}} \put(17.5,45){\circle*{2}}
\put(10,35){\circle*{2}} \put(10,55){\circle*{2}}
\put(10,35){\line(-3,4){7.5}} \put(10,35){\line(-1,4){2.5}}
\put(10,35){\line(3,4){7.5}} \put(10,35){\line(1,4){2.5}}
\put(10,55){\line(-3,-4){7.5}} \put(10,55){\line(-1,-4){2.5}}
\put(10,55){\line(3,-4){7.5}} \put(10,55){\line(1,-4){2.5}}
\put(2.5,45){\line(1,0){15}} \qbezier(2.5,45)(10,37)(17.5,45)

\put(30,35){\circle*{2}} \put(35,40){\circle*{2}}
\put(40,50){\circle*{2}} \put(40,55){\circle*{2}}
\put(45,40){\circle*{2}} \put(50,35){\circle*{2}}
\put(35,40){\line(-1,-1){5}} \put(35,40){\line(1,0){10}}
\put(35,40){\line(1,2){5}} \put(40,50){\line(0,1){5}}
\put(45,40){\line(1,-1){5}} \put(45,40){\line(-1,2){5}}

\put(65,35){\circle*{2}} \put(65,50){\circle*{2}}
\put(80,35){\circle*{2}} \put(80,50){\circle*{2}}
\put(72.5,42.5){\circle*{2}} \put(72.5,55){\circle*{2}}
\put(72.5,55){\line(0,-1){12.5}} \put(72.5,42.5){\line(1,1){7.5}}
\put(72.5,42.5){\line(-1,-1){7.5}} \put(72.5,42.5){\line(-1,1){7.5}}
\put(72.5,42.5){\line(1,-1){7.5}} \put(65,35){\line(1,0){15}}
\put(65,35){\line(0,1){15}} \put(80,50){\line(-1,0){15}}
\put(80,50){\line(0,-1){15}}

\put(95,35){\circle*{2}} \put(95,42.5){\circle*{2}}
\put(102.5,35){\circle*{2}} \put(102.5,50){\circle*{2}}
\put(110,35){\circle*{2}} \put(110,42.5){\circle*{2}}
\put(95,35){\line(0,1){7.5}} \put(110,35){\line(0,1){7.5}}
\put(95,42.5){\line(1,0){15}} \put(95,42.6){\line(1,1){7.5}}
\put(95,42.5){\line(1,-1){7.5}} \put(110,42.5){\line(-1,1){7.5}}
\put(110,42.5){\line(-1,-1){7.5}}

\end{picture}
\end{center}
\caption{A octahedron, a net, a 4-wheel with a pendant adjacent to
the center, and a diamond with a pendant vertex adjacent to each
vertex of degree three}\label{fig:FIS_special_class}
\end{figure}
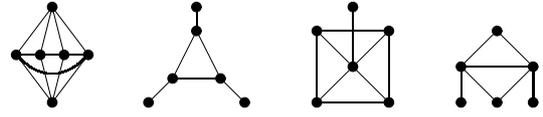

\begin{theorem}
\label{char special}
A graph $G=(V,E)$ is a special 2-threshold graph if and only if
it has no induced $2K_2$, $C_5$, net, house, gem, octahedron, or
a 4-wheel with a pendant vertex adjacent to the center,
or a diamond with a pendant vertex
attached to each vertex of degree three.
\end{theorem}
\begin{proof}
It is easy to check that the listed subgraphs are in the
obstruction set.
We prove the sufficiency.
If a connected graph has no $P_5$ nor $C_5$ then it has a
dominating clique~\cite{kn:bacso,kn:brandstadt}.
That is, it has a clique $C$ such that every vertex outside $C$ has
at least one neighbor in $C$. Since the graph has no $2K_2$ it has
also no $P_5$ and, since there is also no $C_5$
there is a dominating clique.

Since there is no
$2K_2$ there is at most one component with more than one vertex.
If there is a component with more than one vertex then assume that
there is a decomposition tree for that component. Otherwise start with
a decomposition tree that consists of a single leaf and map any of the
isolated vertices to that leaf. Color the vertex black.
We add the isolated vertices one by one as black vertices as follows.
Add a new root to the decomposition tree with a
$\oplus$-operator.
Add the isolated vertex as a right child and the original
root as a left child.

Henceforth assume that $G$ is connected.

Let $C$ be a
dominating clique of maximal cardinality. Among the dominating
cliques of maximal cardinality choose $C$ such that the number of
edges with one endpoint in $C$ and the other endpoint in $V-C$
is maximal.

We claim that $V=C$ or that
$G-C$ is bipartite. For that, it is sufficient to show that there
is no triangle. Assume that there is a triangle $\{x,y,z\}$ in
$V-C$.
Let $X$, $Y$ and $Z$ be the vertices in $C$ that are adjacent to $x$, $y$
and $z$ respectively, but that are not adjacent to any other
vertex of the triangle $\{x,y,z\}$.
Let
\[X_2=(N(y) \cap N(z) \cap C)\setminus N(x)\]
and define $Y_2$ and $Z_2$ similarly. Let $N$ be the set of
vertices in $C$ adjacent to all three of $x$, $y$ and $z$.
Since there is no gem, there are no edges
between
$X$ and $Y_2$, nor between $X$ and $Z_2$, nor between
$Y$ and $X_2$, nor between $Y$ and $Z_2$, nor between $Z$ and $X_2$, nor
between $Z$ and $Y_2$. Since $C$ is a clique
at least one of every pair of these is empty.
Since there is no house, there are no edges
between $X$ and $Y$, nor between $X$ and $Z$, nor between $Y$ and $Z$.
Thus at least one of every pair is empty.
Assume that $Y=Z=\es$.
If $X \neq \es$ then $Y_2=Z_2=\es$.

There is at most one vertex in $C$ which is not adjacent to
$x$ or $y$ or $z$, otherwise there is a $2K_2$. Assume that there
is such a vertex $p$. Then $X=\es$, otherwise there
is a $2K_2$; let $\epsilon \in X$ then the edges
$(\epsilon,p)$ and $(y,z)$ form a
$2K_2$. At least one of $X_2$, $Y_2$ and $Z_2$ is empty
otherwise there is an octahedron. Say that $X_2=\es$.
Obviously, $C \neq \{p\}$ since
$C$ is a dominating clique and $p$ is not adjacent to $x$, $y$ or $z$.
Assume that there is a vertex $p^{\prime} \not\in C$
that is adjacent to $p$ but not adjacent to $x$.
Then $p^{\prime}$ is adjacent to $y$ and to $z$ otherwise
there is a $2K_2$. If $p^{\prime}$ is adjacent to some vertex of
$Y_2$ and to some vertex of $Z_2$ then there is an octahedron.
Assume that $p^{\prime}$ is not adjacent to any vertex of $Y_2$.
If $Y_2 \neq \es$, then there is a gem in $Y_2+\{z,x,y,p^{\prime}\}$.
Thus $Y_2=\es$. If there is a vertex in $Z_2$ that is adjacent
to $p^{\prime}$ then there is a gem contained in
$Z_2+\{p^{\prime},z,y,p\}$. So $p^{\prime}$ is nonadjacent to all
vertices of $Z_2$. If $Z_2\neq \es$, then there is a gem induced by
$Z_2+\{y,p^{\prime},z,x\}$. Thus $Z_2=\es$. It follows that all
vertices of $C-p$ are adjacent to $x$, to $y$ and to $z$.
If some vertex of $N$ is adjacent to $p^{\prime}$ then there is a
gem contained in $N+\{p,p^{\prime},y,x\}$. Thus $p^{\prime}$ is
nonadjacent to all vertices of $N$. Now there is a house
induced by $N+\{y,p^{\prime},p,x\}$, which is a contradiction.

Now assume that all neighbors of $p$ are also adjacent to $x$.
We can replace $p$ in $C$ by $x$. This gives a dominating clique
$C^{\prime}$ of the same cardinality as $C$. The number of
edges with one endpoint in $C$ and the other endpoint in $V-C$
increases since $x$ is furthermore adjacent to $y$ and $z$.

Assume that every vertex in $C$ is adjacent to at least one
vertex of $x$, $y$ and $z$. Assume $X \neq \es$. Then it contains
only one vertex $X=\{p\}$ otherwise there is a $2K_2$.
Furthermore, $Y_2=Z_2=\es$. Consider a neighbor $p^{\prime}\neq x$ of
$p$. If $p^{\prime}$ is not adjacent to $y$ and not adjacent
to $z$ there is a $2K_2$ namely $\{(y,z),(p,p^{\prime})\}$.
If $p^{\prime}$ is adjacent to $y$ and not adjacent to $z$
there is a house or a gem induced by $\{z,p,p^{\prime},x,y\}$.
Thus every neighbor of $p$ is a neighbor of both $y$ and $z$.
Notice that $N+X_2 \neq \es$, otherwise $y$ and $z$ have no
neighbor in $C$. If we replace $p$ in $C$ by $\{y,z\}$ we obtain
a dominating clique of larger cardinality which is a contradiction.

Assume that $X=\es$. Since there is no octahedron at least one
of $X_2$, $Y_2$ and $Z_2$ is empty. Without loss of generality
assume that $X_2=\es$.
Then $C+\{x\}$ is also a (dominating) clique which contradicts the
maximal cardinality of $X$.

This proves the claim that either $V-C$ is empty or that $G[V-C]$
is bipartite.

Since there is no $2K_2$ the bipartite graph is a
difference graph~\cite{kn:hammer2}. Every induced subgraph without isolated
vertices has a vertex in each side of the bipartition that is adjacent to
all the vertices in the other side of the bipartition.

Assume that there is a component with more than one vertex.
There is a unique partition of the vertices into two independent sets.
Call these sets $B$ (black) and $W$ (white). We prove that $G[B+C]$
and $G[W+C]$ are
threshold graphs. Assume that there are two black vertices $x$ and
$y$ that have private neighbors $x^{\prime}$ and $y^{\prime}$ in $C$.
That is, $x^{\prime}$ is a neighbor of $x$ but not of $y$ and
$y^{\prime}$ is a neighbor of $y$ but not of $x$.
Since $x$ and $y$ are in a component of $G[B+W]$ there is an
alternating path
$P=[x,x_1,\ldots,x_k,y]$
of black and white vertices.
Assume $k=1$. If $x_1$ is not adjacent to $x^{\prime}$ and
not adjacent to $y^{\prime}$ there is a $C_5$ which is a contradiction.
If $x_1$ is adjacent to exactly one of $x^{\prime}$ and $y^{\prime}$
then there is a house. If $x_1$ is adjacent to both $x^{\prime}$
and $y^{\prime}$ then there is a gem.

We proceed by induction. Assume that there is a black vertex $x_i$ in $P$
that is adjacent to exactly one of $x^{\prime}$ and $y^{\prime}$.
Then we obtain a contradiction by considering one of the two
subpaths $[x,x_1,\ldots,x_i]$ and $[x_i,\ldots,x_k,y]$. Assume that
every black vertex in $P$, except $x$ and $y$,
is adjacent to both or to neither of
$x^{\prime}$ and $y^{\prime}$. Assume that there exist two white
vertices $x_i$ and $x_j$ in $P$ such that $x_i$ is adjacent to
$x^{\prime}$ but not to $y^{\prime}$ and $x_j$ is adjacent to
$y^{\prime}$ but not to $x^{\prime}$. Since
$x_i$ and $x_j$ are both white they are not adjacent.
Then we obtain a contradiction
by considering the subpath of $P$ with terminals $x_i$ and $x_j$.
So we can assume that all the white vertices in $P$ are adjacent to
both or neither of $x^{\prime}$ and $y^{\prime}$ or, that they
all have the same neighbor in $\{x^{\prime},y^{\prime}\}$.
Assume that all the white vertices are adjacent to $x^{\prime}$
but not to $y^{\prime}$. If the black vertex $x_2$ is not adjacent to
$x^{\prime}$ then
$\{x^{\prime},x,x_1,x_2,x_3\}$ induces a house. If
$x_2$ is adjacent to $x^{\prime}$ then this set induces a gem.
The case where all the white vertices are adjacent to $y^{\prime}$ but
not to $x^{\prime}$ is similar. Assume that all the white vertices
are adjacent to both or neither of $x^{\prime}$ and $y^{\prime}$.
Every edge of $P$ must have at least one endpoint
adjacent to one of $x^{\prime}$ or $y^{\prime}$
otherwise there is a $2K_2$.
Assume that all the white vertices are adjacent to both $x^{\prime}$
and $y^{\prime}$. Then
$\{x^{\prime},x,x_1,x_2,x_3\}$ induces a gem or a house.
Let $x_i$ be the first white vertex in $P$ that is not adjacent to
$x^{\prime}$ nor to $y^{\prime}$. Assume $i=1$.
Then $\{x^{\prime},y^{\prime},x,x_1,x_2\}$ induces a house or there
is a $C_5$ or a $2K_2$. The case where $i=k$ is similar.
Assume that $1 < i < k$. Then $x_{i-2}$, $x_{i-1}$ and
$x_{i+1}$ are adjacent to $x^{\prime}$.
Now $\{x^{\prime},x_{i-2},x_{i-1},x_i,x_{i+1}\}$ induces a house.

This proves the claim that $G[B+C]$ and $G[W+C]$ are threshold graphs.

We prove that the neighborhoods of the black vertices and the
neighborhoods of the white vertices
are ordered by inclusion. Assume that there exist two black vertices
$x$ and $y$ with private neighbors $x^{\prime} \in C$ and
$z \in W$. Since there is no $2K_2$ the vertex $z$ is adjacent
to $x^{\prime}$. Because $C$ is a dominating clique the vertex
$y$ has a neighbor $y^{\prime} \in C$. The vertex $y$ is not
adjacent to $x^{\prime}$ and this implies that
$y^{\prime} \neq x^{\prime}$.
Since the neighborhoods in $C$
of the black vertices are ordered by inclusion and since $y$
is not adjacent to $x^{\prime}$, $x$ is adjacent to $y^{\prime}$.
If $z$ is adjacent to $y^{\prime}$ then
$\{y^{\prime},x,x^{\prime},z,y\}$ induces a gem. If $z$ is not
adjacent to $y^{\prime}$ then this set induces a house.
Similarly it holds true that no two white vertices have private neighbors.

This proves the claim that the neighborhoods of the black
vertices and the neighborhoods of the white vertices are
ordered by inclusion.

Consider a vertex $x \in C$ that has white neighbors and black
neighbors. We prove that every white neighbor is adjacent to
every black neighbor. Let $\alpha$ be a black neighbor and let
$\beta$ be a white neighbor and assume that $\alpha$ and $\beta$
are not adjacent. Since $\alpha$ and $\beta$ are in a component
of $G-C$ there is an alternating path $P=[\alpha,x_1,\ldots,x_k,\beta]$
of black and white vertices.
Since the neighborhoods of the black vertices are ordered
by inclusion there is at most one black vertex $x_i$ on this
path {\em i.e.\/}, $i=2$.
Similarly, there is at most one white vertex $x_j$ on this path,
{\em i.e.\/}, $j=1$.
The black vertex $\alpha$ is not adjacent to $\beta$ and
the black vertex $x_2$ is adjacent to
$\beta$ and either $N(\alpha) \subseteq N(x_2)$
or $N(x_2) \subseteq N(\alpha)$.
This implies that $N(\alpha) \subset N(x_2)$. Thus
$x_2$ is adjacent to $x$. Similarly $x_1$ is adjacent to $x$.
Thus $P+x$ induces a gem, which is a contradiction.

Let $x$ and $y$ be two vertices in $C$ that have neighbors in $B$.
We prove that
\[\text{if~} \es \neq N(y) \cap B \subset N(x) \cap B
\text{~then~} N(x) \cap W \subseteq N(y) \cap W.\] Assume that $x$
has a neighbor $p$ in $B$ that is not adjacent to $y$. Let $q$ be a
neighbor of $y$ in $B$. Let $r$ be a neighbor of $x$ in $W$ that is
not adjacent to $y$. Since $G[B+C]$ is a threshold graph and since
$p$ is not adjacent to $y$ the vertex $q$ is adjacent to $x$. By the
previous observation the vertices $p$ and $q$ are both adjacent to
$r$. This implies that $\{x,y,q,r,p\}$ induces a gem.

Let $C^{\ast} \subseteq C$ be the set of vertices in $C$ that have
no neighbors in $B$ or in $W$. We prove that $C^{\ast}=\es$.
Assume that $C^{\ast} \neq \es$. Then $|C^{\ast}|=1$ otherwise there
is a $2K_2$ since there is an edge in $G[B+W]$. Let $C^{\ast}=\{c\}$.
Let $p$ be a black vertex with a minimal neighborhood and let
$q$ be a white vertex with a minimal neighborhood. Assume that
$p$ and $q$ are not adjacent. Let $p^{\prime} \in C$ be a
neighbor of $p$ and let $q^{\prime} \in C$ be a neighbor
of $q$. Then $p^{\prime} \neq q^{\prime}$ otherwise $p$ and $q$
are adjacent.
Let $a$ be a black vertex with a maximal neighborhood and let
$b$ be a white vertex with a maximal neighborhood. Then $a$
and $b$ are adjacent. Assume that
$a \neq p$ and $b \neq q$. The vertex $a$ is adjacent to $p^{\prime}$
and the vertex $b$ is adjacent to $q^{\prime}$. If $a$ is not adjacent to
$q^{\prime}$ and $b$ is not adjacent to $p^{\prime}$ then
$\{a,b,q^{\prime},p^{\prime},c\}$ induces a house.
If $a$ is adjacent to
$q^{\prime}$ then
$a$ is adjacent to $q$ and
$\{q^{\prime},q,a,p^{\prime},c\}$ induces a gem.
Assume that $a=p$, {\em i.e.\/}, assume that there is only one
black vertex. Then $p$ is adjacent to $q$ because $G[B+W]$ is
connected. This contradicts the assumption that $p$ and $q$ are not
adjacent.
Assume that $p$ and $q$ are adjacent. Then $G[B+W]$ is
complete bipartite.
Assume that $p^{\prime} \neq q^{\prime}$.  If $p$ is not
adjacent to $q^{\prime}$ and $q$ is not adjacent to $p^{\prime}$ then
$\{c,p,q,q^{\prime},p^{\prime}\}$
induces a house. If $p$ is adjacent to $q^{\prime}$ and $q$ is not
adjacent to $p^{\prime}$ then $\{q^{\prime},q,p,p^{\prime},c\}$
induces a gem. Assume that $p$ is adjacent to
$q^{\prime}$ and that $q$ is adjacent to $p^{\prime}$.
Then all the black vertices and all the
white vertices are adjacent to $p^{\prime}$ and
to $q^{\prime}$. Assume that there is a black vertex $a \neq p$
and a white vertex $b \neq q$. Then $\{a,b,p,q,p^{\prime},c\}$
induces a 4-wheel with a pendant.
Assume that there is only one black vertex $p$.
Assume that there is a vertex $\delta \in C-C^{\ast}$ which is
not adjacent to $p$. Then $\delta$ is adjacent to a white
vertex $b$ and $\{p^{\prime},c,\delta,b,p\}$ induces a gem.
Assume that there is an isolated vertex $x$ in $G-C$
adjacent to $c$. Then $\{(p,q),(c,x)\}$ is a
$2K_2$.
Thus $N(c) \subset N(p)$.
Replace $c$ by $p$. Then we obtain a dominating clique $C^{\prime}$
with the same
cardinality as $C$ but there are more edges with one endpoint
in $C^{\prime}$ and the other endpoint in $V-C^{\prime}$ since
$p$ is furthermore adjacent to $q$.
This contradicts the assumption that $C$
maximizes the number of
edges with one endpoint in $C$ and the other in $V-C$.
Assume $p^{\prime}=q^{\prime}$. Then $G[B+W]$ is complete bipartite.
If there are black and white vertices $a \neq p$ and $b \neq q$
then $\{a,b,p,q,p^{\prime},c\}$ induces a 4-wheel with a pendant.
Assume that there is only one black vertex $p$. If there is a vertex
$\delta \in C-C^{\ast}$ that is not adjacent to $p$ then $\delta$
is adjacent to a white vertex $b$. Then $\{p^{\prime},c,\delta,b,p\}$
induces a gem. There is no isolated vertex in $G-C$ adjacent to $c$
otherwise there is a $2K_2$. As above we can replace $c$ by $p$
and obtain a dominating clique $C^{\prime}$
of the same cardinality as $C$ but with
more edges than $C$
with one endpoint in $C^{\prime}$ and the other endpoint
in $V-C^{\prime}$.

This proves the claim that every vertex of $C$ has a neighbor
in $B+W$.

Define an ordering on the vertices of $C+B+W$ as follows.
Let $x \preceq y$ if
\begin{enumerate}[\rm 1.]
\item $x \in B$ and $y \in B$ and
$N(y) \subseteq N(x)$,
\item $x \in B$ and $y \in C$ and
$y \in N(x)$,
\item $x \in B$ and $y \in W$ and
$y \in N(x)$,
\item $x \in C$ and $y \in C$ and
$N(x) \cap B \subseteq N(y) \cap B$ and
$N(y) \cap W \subseteq N(x) \cap W$,
\item $x \in C$ and $y \in W$ and
$y \in N(x)$,
\item $x \in W$ and $y \in W$ and
$N(x) \subseteq N(y)$,
\end{enumerate}
and the inverse conditions for the remaining pairs.
For every pair $x$ and $y$ in $C+B+W$ either $x \preceq y$ or $y \preceq x$.
We proved that the relation is transitive.
It is not necessarily antisymmetric.
Define an equivalence relation on
the vertices as follows. Two vertices
of $C+B+W$ are equivalent if they are in the same set of the
partition and, if they have the same open or closed neighborhood. Then the
relation $\preceq$ induces a linear order on the equivalence classes.
Notice that we could define the reversed order likewise.

Let $I$ be the subset of vertices in $C$ that have no black neighbors.
Let $J$ be the subset of vertices in $C$ that have no white neighbors.
The vertices of $I$ appear before all the black vertices in the linear order
and the vertices of $J$ appear after all the white vertices in the
linear order. Thus $I \cap J =\es$ since $G[B+W]$
is connected. Also $I \neq \es$ and $J \neq \es$ since $C$ is a
maximal clique.

Let $x$ be an isolated vertex in $G-C$. We prove that
$N(x)$ is contained in $I$ or it is contained in $J$.
Let $p$ be a black vertex with a minimal neighborhood and let
$q$ be a white vertex with a minimal neighborhood. Assume that $p$
and $q$ are not adjacent. Let $p^{\prime} \in C$ be a neighbor
of $p$ and let $q^{\prime} \in C$ be a neighbor of $q$.
Then $p^{\prime} \neq q^{\prime}$ and $p$
is not adjacent to $q^{\prime}$ and
$q$ is not adjacent to $p^{\prime}$. Assume that there is an isolated
vertex $x$ in $G-C$ adjacent to $p^{\prime}$ and to $q^{\prime}$.
Let $a$ be a black vertex with a maximal neighborhood and let
$b$ be a white vertex with a maximal neighborhood. Then $a$ and $b$
are adjacent. The vertex $a$ is adjacent to $p^{\prime}$ and
the vertex $b$ is adjacent to $q^{\prime}$. If $a$ is not adjacent to
$q^{\prime}$ and $b$ is not adjacent to $p^{\prime}$ then there is a
house $\{a,b,q^{\prime},p^{\prime},x\}$.
If $a$ is adjacent to $q^{\prime}$
then $a$ is adjacent to $q$ and there is a gem
$\{q^{\prime},q,a,p^{\prime},x\}$.
If there is only one black vertex $p$ then $p$ is adjacent to $q$
since $G[B+W]$ is connected. This contradicts the assumption that
$p$ and $q$ are not adjacent.
Assume that $p$ and $q$ are adjacent.
Then $G[B+W]$ is complete bipartite.
Assume that
$p^{\prime} \neq q^{\prime}$.
If $p$ is not adjacent to
$q^{\prime}$ and $q$ is not adjacent to $p^{\prime}$ then
$\{p,q,p^{\prime},q^{\prime},x\}$ induces a house. If $p$ is
adjacent to $q^{\prime}$ but $q$ is not adjacent to $p^{\prime}$
then $\{q^{\prime},q,p,p^{\prime},x\}$ induces a gem.
Assume that $p$ is adjacent to $q^{\prime}$ and that $q$ is
adjacent to $p^{\prime}$.
There is a vertex $\delta_1 \in C$ that is
adjacent to $p$ but not to $q$ and there is a vertex
$\delta_2 \in C$ that is adjacent to $q$ but not to $p$.
If $\delta_1$ is adjacent to $x$ then there is a gem
$\{p^{\prime},x,\delta_1,p,q\}$. If $\delta_2$ is adjacent to $x$
then there is a gem $\{q^{\prime},x,\delta_2,q,p\}$.
If $\delta_1$ is not adjacent to $x$ and $\delta_2$ is not adjacent to
$x$ then there is a 4-wheel with a pendant attached to its center
$\{p,\delta_1,\delta_2,q,p^{\prime},x\}$.
The same contradiction follows when $p^{\prime}=q^{\prime}$.

Assume that $x$ is not adjacent to any neighbor of $p$ and that
$x$ is not adjacent to any neighbor of $q$.
Assume that $p$ and $q$ are not adjacent. Let $p^{\prime} \in C$
be a neighbor of $p$ and let $q^{\prime} \in C$ be
a neighbor of $q$. Let $\delta \in C$ be a neighbor of $x$.
Then $\{p,q,x,p^{\prime},q^{\prime},\delta\}$ is a net.
Assume that $p$ and $q$ are adjacent.
Let $\delta_1 \in C$ be a neighbor of $p$ that is not a
neighbor of $q$ and let $\delta_2 \in C$ be a neighbor of
$q$ that is not a neighbor of $p$. Let $\delta \in C$
be a neighbor of $x$. Then $\{\delta,\delta_1,\delta_2,q,p\}$
induces a house.
This proves that $x$ has common neighbors in $C$ with exactly one
of $p$ and $q$. Assume that $x$ has no common neighbors with $q$.
We proceed by induction on the number of white vertices. Remove
$q$. By induction it follows that $x$ has no common neighbors
with any white vertex.

This proves the claim that $N(x)$ is contained in $I$ or in $J$
for any isolated vertex $x$ of $G-C$.

We prove that the neighborhoods of
isolated vertices of $G-C$ that have
neighbors in $I$ are ordered by set inclusion.
Consider three isolated vertices $x$, $y$ and $z$ of $G-C$
that have neighbors in $I$.
Assume that $x$ and $y$ have private neighbors $x^{\prime}$
and $y^{\prime}$ in $C$. Assume that $z$ is adjacent to $x^{\prime}$
and to $y^{\prime}$. There exists a vertex $\delta \in C$
that is not adjacent to $z$. If $x$ is adjacent to $\delta$ there
is a gem $\{x^{\prime},x,\delta,y^{\prime},z\}$. If
$y$ is adjacent to $\delta$ there is a
gem $\{y^{\prime},y,\delta,x^{\prime},z\}$. If $\delta$ is not adjacent to
$x$ and not to $y$ there is a diamond with two pendant vertices
$\{z,x^{\prime},y^{\prime},\delta,x,y\}$.
Assume that $z$ is not adjacent to $x^{\prime}$ and that
$z$ is not adjacent to $y^{\prime}$. Assume that $z$ has a neighbor
$z^{\prime}$ that is not adjacent to $x$ and that is not adjacent to
$y$. Then $\{x^{\prime},y^{\prime},z^{\prime},x,y,z\}$ is a net.
If $z$ has a neighbor $z_1$ that is adjacent to $x$ but not to
$y$ and a neighbor $z_2$ that is adjacent to $y$ but not to $x$
then there is a gem $\{z_1,x,x^{\prime},z_2,z\}$.
Assume that $z$ is adjacent to $x^{\prime}$ but not adjacent to
$y^{\prime}$. Assume that $x$ has a neighbor $x^{\prime\prime}$
that is not adjacent to $z$ and that $z$ has a neighbor $z^{\prime}$
that is not adjacent to $x$. Then
$\{x^{\prime},z,z^{\prime},x^{\prime\prime},x\}$ is a gem.
Thus the neighborhood
of $z$ is comparable with the neighborhood of exactly one of
$x$ and $y$ and it is disjoint with the neighborhood of the
other one.

This proves the claim that the neighborhoods of the isolated vertices
with neighbors in $I$ are ordered by set inclusion.
A similar proof shows that the neighborhoods of the white vertices
and of the isolated vertices with their neighborhoods in $I$ and,
the neighborhoods of the black vertices and of the isolated
vertices with their neighborhoods in $J$ are ordered by set inclusion.

Color the vertices of $C$ white and equip them with a $\otimes_w$-operator.
If there is a component in $G-C$ with more than one vertex then
color the vertices of $B$ black and equip them with a $\otimes_w$-operator.
Color the vertices of $W$ white and equip them with a $\oplus$-operator.
Color the isolated vertices of $G-C$ that have their neighbors in
$I$ white and equip them with a $\oplus$-operator. Color the isolated
vertices of $G-C$ that have their neighbors in $J$ black and equip them
with a $\otimes_w$-operator. If $G-C$ has no component with more
than one vertex then let $I$ and $J$ be maximal disjoint neighborhoods
in $C$. We proved that there is a linear ordering of the vertices
that satisfies the neighborhood conditions.

This proves the theorem. \end{proof}

\begin{remark}
A class of graphs that is closely related to the special 2-threshold
graphs is the class of probe threshold graphs~\cite{kn:bayer,kn:chandler}.
The obstruction set for probe threshold graphs appears in~\cite{kn:bayer}
(without proof). That the classes are different is illustrated by
the complement of $2P_3$. This graph is not a probe threshold graph but it
is a special 2-threshold graph. The gem is an example of a probe threshold
graphs that is not a special 2-threshold graph.
\end{remark}

\begin{lemma}
\label{rw 1}
The special 2-threshold graphs have rankwidth one. They can be
recognized in linear time.
\end{lemma}
\begin{proof}
A graph $G$ is distance hereditary if in every induced subgraph
$G[W]$ for every pair of vertices in a component of $G[W]$ the
distance between them is the same as in $G$~\cite{kn:howorka}. The
distance-hereditary graphs can be characterized as follows. Every
induced subgraph has an isolated vertex, or a pendant vertex, or two
vertices $x$ and $y$ such that every other vertex $z$ is adjacent to
both $x$ and $y$ or to neither of them. A pair of vertices like that
is called a twin. It is easy to check that the special graphs
satisfy this property. The distance-hereditary graphs are exactly
the graphs of rankwidth one~\cite{kn:oum4}. A rank-decomposition
tree can be found in linear time~\cite{kn:hammer}. We can formulate
the existence of an induced subgraph that is isomorphic to one of
the graphs in the obstruction set in monadic second-order logic
without quantification over edge-sets. This proves the lemma.
\end{proof}

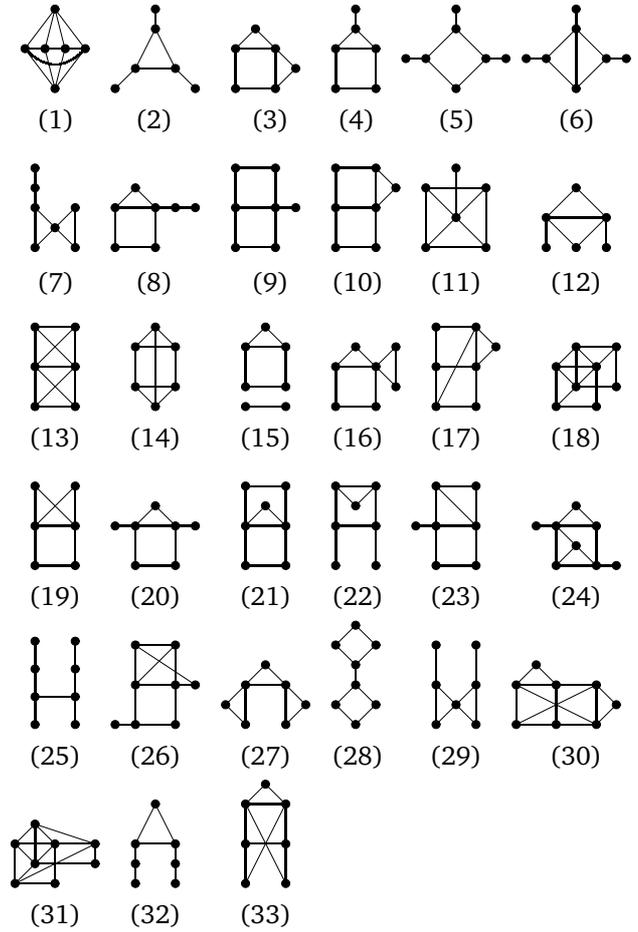
\begin{figure}[t]
\setlength{\unitlength}{1.5pt}
\begin{center}
\begin{picture}(152.5,235)(0,-160)
\put(2.5,60){\circle*{2}} \put(7.5,60){\circle*{2}}
\put(12.5,60){\circle*{2}} \put(17.5,60){\circle*{2}}
\put(10,50){\circle*{2}} \put(10,70){\circle*{2}}
\put(10,50){\line(-3,4){7.5}} \put(10,50){\line(-1,4){2.5}}
\put(10,50){\line(3,4){7.5}} \put(10,50){\line(1,4){2.5}}
\put(10,70){\line(-3,-4){7.5}} \put(10,70){\line(-1,-4){2.5}}
\put(10,70){\line(3,-4){7.5}} \put(10,70){\line(1,-4){2.5}}
\put(2.5,60){\line(1,0){15}} \qbezier(2.5,60)(10,52)(17.5,60)
\put(5.5,40){(1)}

\put(25,50){\circle*{2}} \put(30,55){\circle*{2}}
\put(35,65){\circle*{2}} \put(35,70){\circle*{2}}
\put(40,55){\circle*{2}} \put(45,50){\circle*{2}}
\put(30,55){\line(-1,-1){5}} \put(30,55){\line(1,0){10}}
\put(30,55){\line(1,2){5}} \put(35,65){\line(0,1){5}}
\put(40,55){\line(1,-1){5}} \put(40,55){\line(-1,2){5}}
\put(30,40){(2)}

\put(55,50){\circle*{2}} \put(55,60){\circle*{2}}
\put(60,65){\circle*{2}} \put(65,50){\circle*{2}}
\put(65,60){\circle*{2}} \put(70,55){\circle*{2}}
\put(55,50){\line(1,0){10}} \put(55,50){\line(0,1){10}}
\put(65,60){\line(-1,0){10}} \put(65,60){\line(0,-1){10}}
\put(60,65){\line(-1,-1){5}} \put(60,65){\line(1,-1){5}}
\put(70,55){\line(-1,1){5}} \put(70,55){\line(-1,-1){5}}
\put(59,40){(3)}


\put(80,50){\circle*{2}} \put(80,60){\circle*{2}}
\put(85,65){\circle*{2}} \put(90,50){\circle*{2}}
\put(85,70){\circle*{2}} \put(90,60){\circle*{2}}
\put(80,50){\line(1,0){10}} \put(80,50){\line(0,1){10}}
\put(90,60){\line(-1,0){10}} \put(90,60){\line(0,-1){10}}
\put(85,65){\line(-1,-1){5}} \put(85,65){\line(1,-1){5}}
\put(85,70){\line(0,-1){5}} \put(80.5,40){(4)}


\put(110,50){\circle*{2}} \put(102.5,57.5){\circle*{2}}
\put(110,65){\circle*{2}} \put(117.5,57.5){\circle*{2}}
\put(97.5,57.5){\circle*{2}} \put(122.5,57.5){\circle*{2}}
\put(110,70){\circle*{2}} \put(110,50){\line(1,1){7.5}}
\put(110,50){\line(-1,1){7.5}} \put(110,65){\line(-1,-1){7.5}}
\put(110,65){\line(1,-1){7.5}} \put(97.5,57.5){\line(1,0){5}}
\put(122.5,57.5){\line(-1,0){5}} \put(110,70){\line(0,-1){5}}
\put(105.5,40){(5)}


\put(140,50){\circle*{2}} \put(132.5,57.5){\circle*{2}}
\put(140,65){\circle*{2}} \put(147.5,57.5){\circle*{2}}
\put(127.5,57.5){\circle*{2}} \put(152.5,57.5){\circle*{2}}
\put(140,70){\circle*{2}} \put(140,50){\line(1,1){7.5}}
\put(140,50){\line(-1,1){7.5}} \put(140,65){\line(-1,-1){7.5}}
\put(140,65){\line(1,-1){7.5}} \put(127.5,57.5){\line(1,0){5}}
\put(152.5,57.5){\line(-1,0){5}} \put(140,70){\line(0,-1){5}}
\put(140,50){\line(0,1){15}}

\put(135.5,40){(6)}



\multiput(5,10)(0,10){3}{\circle*{2}} \put(5,25){\circle*{2}}
\multiput(15,10)(0,10){2}{\circle*{2}}
\multiput(5,10)(10,0){2}{\line(0,1){10}} \put(5,30){\line(0,-1){10}}
\put(10,15){\circle*{2}} \put(5,10){\line(1,1){10}}
\put(5,20){\line(1,-1){10}} \put(5.5,-1){(7)}


\put(25,10){\circle*{2}} \put(35,10){\circle*{2}}
\put(25,20){\circle*{2}} \put(35,20){\circle*{2}}
\put(30,25){\circle*{2}} \put(40,20){\circle*{2}}
\put(45,20){\circle*{2}} \put(25,10){\line(1,0){10}}
\put(25,10){\line(0,1){10}} \put(35,20){\line(-1,0){10}}
\put(35,20){\line(0,-1){10}} \put(30,25){\line(-1,-1){5}}
\put(30,25){\line(1,-1){5}} \put(35,20){\line(1,0){10}}
\put(30,-1){(8)}

\multiput(55,10)(0,10){3}{\circle*{2}}
\multiput(65,10)(0,10){3}{\circle*{2}} \put(70,20){\circle*{2}}
\multiput(55,10)(0,10){3}{\line(1,0){10}}
\multiput(55,10)(10,0){2}{\line(0,1){20}} \put(65,20){\line(1,0){5}}
\put(59,-1){(9)}

\multiput(80,10)(0,10){3}{\circle*{2}}
\multiput(90,10)(0,10){3}{\circle*{2}} \put(95,25){\circle*{2}}
\multiput(80,10)(0,10){3}{\line(1,0){10}}
\multiput(80,10)(10,0){2}{\line(0,1){20}}
\put(95,25){\line(-1,1){5}} \put(95,25){\line(-1,-1){5}}
\put(79,-1){(10)}

\put(102.5,10){\circle*{2}} \put(102.5,25){\circle*{2}}
\put(117.5,10){\circle*{2}} \put(117.5,25){\circle*{2}}
\put(110,17.5){\circle*{2}} \put(110,30){\circle*{2}}
\put(110,30){\line(0,-1){12.5}} \put(110,17.5){\line(1,1){7.5}}
\put(110,17.5){\line(-1,-1){7.5}} \put(110,17.5){\line(-1,1){7.5}}
\put(110,17.5){\line(1,-1){7.5}} \put(102.5,10){\line(1,0){15}}
\put(102.5,10){\line(0,1){15}} \put(117.5,25){\line(-1,0){15}}
\put(117.5,25){\line(0,-1){15}}

\put(103.5,-1){(11)}

\put(132.5,10){\circle*{2}} \put(132.5,17.5){\circle*{2}}
\put(140,10){\circle*{2}} \put(140,25){\circle*{2}}
\put(147.5,10){\circle*{2}} \put(147.5,17.5){\circle*{2}}
\put(132.5,10){\line(0,1){7.5}} \put(147.5,10){\line(0,1){7.5}}
\put(132.5,17.5){\line(1,0){15}} \put(132.5,17.5){\line(1,1){7.5}}
\put(132.5,17.5){\line(1,-1){7.5}}
\put(147.5,17.5){\line(-1,1){7.5}}
\put(147.5,17.5){\line(-1,-1){7.5}}

\put(133.5,-1){(12)}



\multiput(5,-30)(0,10){3}{\circle*{2}}
\multiput(15,-30)(0,10){3}{\circle*{2}}
\multiput(5,-30)(0,10){3}{\line(1,0){10}}
\multiput(5,-30)(10,0){2}{\line(0,1){20}}
\multiput(5,-30)(0,10){2}{\line(1,1){10}}
\multiput(5,-20)(0,10){2}{\line(1,-1){10}} \put(3.5,-40){(13)}


\put(30,-25){\circle*{2}} \put(40,-25){\circle*{2}}
\put(30,-15){\circle*{2}} \put(40,-15){\circle*{2}}
\put(35,-10){\circle*{2}} \put(35,-30){\circle*{2}}
\put(30,-25){\line(1,0){10}} \put(30,-25){\line(0,1){10}}
\put(30,-15){\line(1,0){10}} \put(40,-25){\line(0,1){10}}
\put(35,-10){\line(-1,-1){5}} \put(35,-10){\line(1,-1){5}}
\put(35,-30){\line(1,1){5}} \put(35,-30){\line(-1,1){5}}
\put(35,-30){\line(0,1){20}}

\put(28.5,-40){(14)}


\put(57.5,-25){\circle*{2}} \put(57.5,-15){\circle*{2}}
\put(62.5,-10){\circle*{2}} \put(67.5,-25){\circle*{2}}
\put(67.5,-15){\circle*{2}} \put(57.5,-25){\line(1,0){10}}
\put(57.5,-25){\line(0,1){10}} \put(67.5,-15){\line(-1,0){10}}
\put(67.5,-15){\line(0,-1){10}} \put(62.5,-10){\line(-1,-1){5}}
\put(62.5,-10){\line(1,-1){5}}

\put(57.5,-30){\circle*{2}} \put(67.5,-30){\circle*{2}}
\put(57.5,-30){\line(1,0){10}} \put(56,-40){(15)}

\put(80,-30){\circle*{2}} \put(90,-30){\circle*{2}}
\put(80,-20){\circle*{2}} \put(90,-20){\circle*{2}}
\put(85,-15){\circle*{2}} \put(95,-25){\circle*{2}}
\put(95,-15){\circle*{2}} \put(80,-30){\line(1,0){10}}
\put(80,-30){\line(0,1){10}} \put(90,-20){\line(-1,0){10}}
\put(90,-20){\line(0,-1){10}} \put(85,-15){\line(-1,-1){5}}
\put(85,-15){\line(1,-1){5}} \put(90,-20){\line(1,1){5}}
\put(90,-20){\line(1,-1){5}} \put(95,-25){\line(0,1){10}}
\put(79,-40){(16)}


\multiput(105,-30)(0,10){3}{\circle*{2}}
\multiput(115,-30)(0,10){3}{\circle*{2}} \put(120,-15){\circle*{2}}
\multiput(105,-30)(0,10){3}{\line(1,0){10}}
\multiput(105,-30)(10,0){2}{\line(0,1){20}}
\put(115,-10){\line(-1,-2){10}} \put(120,-15){\line(-1,1){5}}
\put(120,-15){\line(-1,-1){5}} \put(103.5,-40){(17)}


\put(135,-30){\circle*{2}} \put(145,-30){\circle*{2}}
\put(135,-20){\circle*{2}} \put(145,-20){\circle*{2}}
\put(140,-15){\circle*{2}} \put(140,-25){\circle*{2}}
\put(150,-15){\circle*{2}} \put(135,-30){\line(1,0){10}}
\put(135,-30){\line(0,1){10}} \put(145,-20){\line(-1,0){10}}
\put(145,-20){\line(0,-1){10}} \put(140,-15){\line(-1,-1){5}}
\put(140,-15){\line(1,-1){5}} \put(140,-25){\line(-1,1){5}}
\put(140,-25){\line(1,1){5}} \put(140,-25){\line(-1,-1){5}}
\put(140,-25){\line(0,1){10}} \put(150,-15){\line(-1,0){10}}
\put(150,-15){\line(-1,-1){5}} \put(150,-25){\circle*{2}}
\put(150,-25){\line(0,1){10}} \put(150,-25){\line(-1,0){10}}

\put(133.5,-40){(18)}



\multiput(5,-70)(0,10){3}{\circle*{2}}
\multiput(15,-70)(0,10){3}{\circle*{2}}
\multiput(5,-70)(10,0){2}{\line(0,1){20}}
\multiput(5,-70)(0,10){2}{\line(1,0){10}}
\put(5,-50){\line(1,-1){10}} \put(15,-50){\line(-1,-1){10}}
\put(3.5,-80){(19)}

\put(30,-70){\circle*{2}} \put(40,-70){\circle*{2}}
\put(30,-60){\circle*{2}} \put(40,-60){\circle*{2}}
\put(25,-60){\circle*{2}} \put(45,-60){\circle*{2}}
\put(35,-55){\circle*{2}} \put(30,-70){\line(1,0){10}}
\put(30,-70){\line(0,1){10}} \put(40,-60){\line(-1,0){10}}
\put(40,-60){\line(0,-1){10}} \put(35,-55){\line(-1,-1){5}}
\put(35,-55){\line(1,-1){5}} \put(25,-60){\line(1,0){5}}
\put(45,-60){\line(-1,0){5}} \put(28.5,-80){(20)}


\multiput(57.5,-70)(0,10){3}{\circle*{2}}
\multiput(67.5,-70)(0,10){3}{\circle*{2}}
\multiput(57.5,-70)(0,10){3}{\line(1,0){10}}
\multiput(57.5,-70)(10,0){2}{\line(0,1){20}}
\put(62.5,-55){\circle*{2}} \put(62.5,-55){\line(-1,-1){5}}
\put(62.5,-55){\line(1,-1){5}} \put(56,-80){(21)}


\multiput(80,-70)(0,10){3}{\circle*{2}}
\multiput(90,-70)(0,10){3}{\circle*{2}}
\multiput(80,-70)(10,0){2}{\line(0,1){20}}
\multiput(80,-60)(0,10){2}{\line(1,0){10}} \put(85,-55){\circle*{2}}
\put(85,-55){\line(1,1){5}} \put(85,-55){\line(-1,1){5}}
\put(79,-80){(22)}


\multiput(105,-70)(0,10){3}{\circle*{2}}
\multiput(115,-70)(0,10){3}{\circle*{2}} \put(100,-60){\circle*{2}}
\multiput(105,-70)(0,10){3}{\line(1,0){10}}
\multiput(105,-70)(10,0){2}{\line(0,1){20}}
\put(100,-60){\line(1,0){5}} \put(105,-50){\line(1,-1){10}}
\put(103.5,-80){(23)}

\put(135,-70){\circle*{2}} \put(135,-60){\circle*{2}}
\put(145,-70){\circle*{2}} \put(145,-60){\circle*{2}}
\put(140,-65){\circle*{2}} \put(140,-55){\circle*{2}}
\put(130,-60){\circle*{2}} \put(150,-70){\circle*{2}}
\put(135,-70){\line(1,0){10}} \put(135,-70){\line(0,1){10}}
\put(145,-60){\line(-1,0){10}} \put(145,-60){\line(0,-1){10}}
\put(130,-60){\line(1,0){5}} \put(150,-70){\line(-1,0){5}}
\put(140,-55){\line(-1,-1){5}} \put(140,-55){\line(1,-1){5}}
\put(140,-65){\line(-1,-1){5}} \put(140,-65){\line(1,-1){5}}
\put(140,-65){\line(-1,1){5}} \put(133.5,-80){(24)}



\multiput(5,-110)(0,7){4}{\circle*{2}}
\multiput(15,-110)(0,7){4}{\circle*{2}} \put(5,-110){\line(0,1){21}}
\put(15,-110){\line(0,1){21}} \put(5,-103){\line(1,0){10}}

\put(3.5,-120){(25)}


\multiput(30,-110)(0,10){3}{\circle*{2}}
\multiput(40,-110)(0,10){3}{\circle*{2}}
\multiput(30,-110)(0,10){3}{\line(1,0){10}}
\multiput(30,-110)(10,0){2}{\line(0,1){20}}
\put(40,-90){\line(-1,-1){10}} \put(25,-110){\circle*{2}}
\put(25,-110){\line(1,0){5}} \put(45,-100){\circle*{2}}
\put(45,-100){\line(-1,0){5}} \put(45,-100){\line(-3,2){15}}
\put(28.5,-120){(26)}


\multiput(57.5,-110)(10,0){2}{\circle*{2}}
\multiput(52.5,-105)(20,0){2}{\circle*{2}}
\multiput(57.5,-100)(10,0){2}{\circle*{2}}
\put(62.5,-95){\circle*{2}} \put(62.5,-95){\line(-1,-1){5}}
\put(62.5,-95){\line(1,-1){5}} \put(57.5,-100){\line(0,-1){10}}
\put(57.5,-100){\line(1,0){10}} \put(67.5,-100){\line(0,-1){10}}
\put(52.5,-105){\line(1,1){5}} \put(52.5,-105){\line(1,-1){5}}
\put(72.5,-105){\line(-1,1){5}} \put(72.5,-105){\line(-1,-1){5}}
\put(56,-120){(27)}


\put(85,-110){\circle*{2}} \put(80,-105){\circle*{2}}
\put(90,-105){\circle*{2}} \put(85,-100){\circle*{2}}
\put(85,-95){\circle*{2}} \put(80,-90){\circle*{2}}
\put(90,-90){\circle*{2}} \put(85,-85){\circle*{2}}
\put(85,-85){\line(-1,-1){5}} \put(85,-85){\line(1,-1){5}}
\put(85,-95){\line(1,1){5}} \put(85,-95){\line(-1,1){5}}
\put(85,-95){\line(0,-1){5}} \put(85,-100){\line(-1,-1){5}}
\put(85,-100){\line(1,-1){5}} \put(85,-110){\line(-1,1){5}}
\put(85,-110){\line(1,1){5}}

\put(79,-120){(28)}


\multiput(105,-110)(0,10){3}{\circle*{2}}
\multiput(115,-110)(0,10){3}{\circle*{2}}
\multiput(105,-110)(10,0){2}{\line(0,1){20}}
\put(110,-105){\circle*{2}} \put(105,-110){\line(1,1){10}}
\put(105,-100){\line(1,-1){10}} \put(103.5,-120){(29)}

\multiput(125,-110)(10,0){3}{\circle*{2}}
\multiput(125,-100)(10,0){3}{\circle*{2}} \put(130,-95){\circle*{2}}
\put(150,-105){\circle*{2}}
\multiput(125,-110)(10,0){3}{\line(0,1){10}}
\multiput(125,-110)(0,10){2}{\line(1,0){20}}
\put(130,-95){\line(-1,-1){5}} \put(130,-95){\line(1,-1){5}}
\put(150,-105){\line(-1,1){5}} \put(150,-105){\line(-1,-1){5}}
\put(125,-110){\line(2,1){20}} \put(125,-100){\line(2,-1){20}}
\put(133.5,-120){(30)}



\put(0,-150){\circle*{2}} \put(10,-150){\circle*{2}}
\put(0,-140){\circle*{2}} \put(10,-140){\circle*{2}}
\put(5,-135){\circle*{2}} \put(5,-145){\circle*{2}}
\put(20,-140){\circle*{2}} \put(20,-145){\circle*{2}}
\put(0,-150){\line(1,0){10}} \put(0,-150){\line(0,1){10}}
\put(10,-140){\line(-1,0){10}} \put(10,-140){\line(0,-1){10}}
\put(5,-135){\line(-1,-1){5}} \put(5,-135){\line(1,-1){5}}
\put(5,-145){\line(-1,-1){5}} \put(5,-145){\line(1,1){5}}
\put(5,-145){\line(-1,1){5}} \put(5,-145){\line(0,1){10}}
\put(20,-140){\line(-1,0){10}} \put(20,-140){\line(-3,1){15}}
\put(20,-145){\line(0,1){5}} \put(20,-145){\line(-1,0){15}}
\put(20,-140){\line(-2,-1){20}}

\put(3.5,-160){(31)}


\multiput(30,-150)(0,5){3}{\circle*{2}}
\multiput(40,-150)(0,5){3}{\circle*{2}} \put(35,-130){\circle*{2}}
\put(30,-140){\line(1,2){5}} \put(40,-140){\line(-1,2){5}}
\put(30,-140){\line(1,0){10}} \put(30,-140){\line(0,-1){10}}
\put(40,-140){\line(0,-1){10}} \put(28.5,-160){(32)}


\multiput(57.5,-150)(0,10){3}{\circle*{2}}
\multiput(67.5,-150)(0,10){3}{\circle*{2}}
\multiput(57.5,-150)(10,0){2}{\line(0,1){20}}
\multiput(57.5,-140)(0,10){2}{\line(1,0){10}}
\put(62.5,-125){\circle*{2}} \put(62.5,-125){\line(-1,-1){5}}
\put(62.5,-125){\line(1,-1){5}} \put(57.5,-150){\line(1,2){10}}
\put(67.5,-150){\line(-1,2){10}}

\put(56,-160){(33)}

\end{picture}
\end{center}
\vspace*{-0.4cm} \caption{The set of forbidden induced subgraphs of
2-threshold graphs contains the gem, $3K_2$, $C_5$, $C_6$, $C_7$,
$P_7$, $2C_4$, $2P_4$, $C_4+P_4$ and the graphs listed above.}
\label{fig:FIS_2_threshold_graphs}
\end{figure}

\section{Partitioned black-and-white threshold graphs}

In this section we turn our attention to the partitioned
2-threshold graphs, that is, we assume that the coloring is a
part of the input.

\begin{lemma}
If $G$ is 2-threshold then the following statements hold true.
\begin{enumerate}[\rm (a)]
\item $G$ has at most two components with at least two vertices.
\item If $G$ has two components with at least two
vertices, then one is a threshold graph and the other is a special
2-threshold graph.
\end{enumerate}
\end{lemma}
\begin{proof}
Assume that $G$ has more than two components with at least
two vertices. Then $G$ has an induced $3K_2$. It is easy to check
that $3K_2$ is not a 2-threshold graph.

Assume that there are two components $C_1$ and $C_2$
with at least two vertices.
Let $x$ and $y$ be adjacent vertices in $C_1$.
First assume that
$x$ is black and $y$ is white.
Assume that $C_2$ has a black vertex and a white vertex.
Then it has an edge with one black endpoint and one
white endpoint. We find a $2K_2$ and both edges have one
black endpoint and one white endpoint. It is easy to check
that there is no elimination ordering for a $2K_2$ that is
colored like that.
Thus $C_2$ is monochromatic which implies that
$C_2$ induces a threshold graph.

Now assume that $x$ and $y$ are both black. Consider a decomposition
tree for $C_2+\{x,y\}$. Assume that $x$ is closer to the root than
$y$. Then the operator of $x$ is a $\otimes_b$-operator since $x$ is
adjacent to $y$. All the operators that appear closer to the root
than the operator of $y$ are either $\oplus$ or $\otimes_w$
otherwise they are adjacent to $y$. All the vertices further from
the root than $y$ are white, otherwise they are adjacent to $x$. So
we may replace any $\otimes_b$-operator that is further from the
root than the operator of $y$ by a $\oplus$-operator. Thus the tree
for $C_2$ now only has $\oplus$- and $\otimes_w$-operators; {\em
i.e.\/}, the graph induced by $C_2$ is a special threshold graph.
\end{proof}

Let $G_1=(V_1,E_1)$ and $G_2=(V_2,E_2)$ be two graphs.
The union of $G_1$ and $G_2$ is the graph
\[G_1 + G_2=(V_1+V_2,E_1+E_2).\]
The join of $G_1$ and $G_2$ is the graph obtained from the
union by adding an edge between every vertex of $V_1$ and
every vertex of $V_2$. We denote the join by $G_1 \times G_2$.

\begin{lemma}
\label{neighbors}
Let $G$ be a 2-threshold graph. The graph induced by every
nonempty neighborhood is one of the following.
\begin{enumerate}[\rm (a)]
\item a threshold graph, or
\item the union of two threshold graphs, or
\item the join of two threshold graphs.
\end{enumerate}
\end{lemma}
\begin{proof}
Consider a black-and-white coloring of the vertices and a
decomposition tree for $G$. Let $x$ and $y$ be two vertices.
We write $x \prec y$ if $x$ is closer to the root than $y$
in the decomposition tree. We write $o(x)$ for the operator
in the tree that is adjacent to $x$.

Let $x$ be a vertex and assume that $x$ is
black. We consider three cases.
\begin{description}
\item[\rm Case 1: $o(x)=\otimes_b$.]
In this case $N(x)=T_1+T_2$ where $T_1$ and $T_2$ are defined as follows.
\begin{eqnarray*}
 T_1&=&\{\; y \in V \;|\; y \prec x \;\;\mbox{and $o(y)=\otimes_b$}\;\}\\
T_2 &=& \{\;y \in V\;|\; x \prec y \;\;\mbox{and $y$ is black}\;\}.
\end{eqnarray*}
Notice that every vertex of $T_1$ is adjacent to every vertex of
$T_2$. All vertices of $T_2$ are black thus $T_2$ induces a threshold graph.
The black vertices of $T_1$ are a clique and the white vertices
are an independent set. Every white vertex is adjacent to those black
vertices that are further from the root. This implies that $T_1$ is
a threshold graph (if it is nonempty)~\cite{kn:mahadev}.
\item[\rm Case 2: $o(x)=\oplus$.]
In this case the neighbors of $x$ are the vertices of $T_1$ described
above.
\item[\rm Case 3: $o(x)=\otimes_w$.]
Now $N(x)=T_1+T_2$ where $T_1$ and $T_2$ are defined as follows.
\begin{eqnarray*}
T_1&=&\{\; y \in V \;|\; y \prec x \;\;\mbox{and $o(y)=\otimes_b$}\;\}\\
T_2&=&\{\;y \in V\;|\; x \prec y \;\;\mbox{and $y$ is white}\;\}.
\end{eqnarray*}
In this case there are no edge between vertices of $T_1$ and
vertices of $T_2$. If $T_2 \neq \es$ then the graph induced by $T_2$
is a threshold graph since all vertices are white.
We proved above that $T_1$ is also a threshold graph.
\end{description}
A similar analysis holds when $x$ is a white vertex. This proves the
lemma. \end{proof}

The complement of a graph $G=(V,E)$ is the graph that has $V$
as its vertices and that has those pairs of vertices adjacent that
are not adjacent in $G$. We denote the complement of $G$ by $\Bar{G}$.

Let $G$ be a graph. A vertex of $G$ is good if its neighborhood
is empty or, a threshold graph or, the union of two threshold graphs
or the join of two threshold graphs. The graph $G$ is good if
every vertex of $G$ is good. By Lemma~\ref{neighbors}
a 2-threshold graph is good. Consider a vertex $x$ in $G$.
A local complementation at $x$ is the operation which complements the
graph induced by $N(x)$. Note that the class of good graphs is
hereditary and closed under local complementations.

\begin{theorem}
\label{char good}
A graph $G$ is good if and only if
$G$ has no induced gem or, $C_4+K_1$ with a universal vertex or,
the local complement of this with
respect to the universal vertex or, $3K_2$ with a universal vertex
or, an octahedron with a universal vertex.
\end{theorem}
\begin{proof}
It is easy to check that the universal vertex of each of
the mentioned graphs is not good. We prove the converse.
Let $x$ be a vertex with a nonempty neighborhood. Since
there is no gem the graph induced by $N(x)$ is a cograph.
Assume that this graph is disconnected. Let $C_1,\ldots,C_t$
be the components of the graph induced by $N(x)$.
There is no $3K_2$ in the graph induced by $N(x)$ so there
are at most two components with more than one vertex.
Since there is no 4-wheel with a pendant vertex attached to its
center the graph induced by each component $C_i$ has
no $C_4$. Assume that $C_i$ has an induced $2K_2$. Since
the graph induced by $C_i$ is connected and since there is
no $P_4$ in $G[C_i]$ this graph has a butterfly. Then $G$
has an induced subgraph that is isomorphic a butterfly
with a universal vertex. The butterfly is the complement
of $C_4+K_1$.

Now assume that the graph induced by $N(x)$ is connected. Consider
the graph obtained by complementing $N(x)$. By the argument above
$x$ is good in this graph. This implies that $x$ is also good in
$G$. \end{proof}

\begin{figure}[t]
\setlength{\unitlength}{1.5pt}
\begin{center}
\begin{picture}(152.5,116)(2.5,-51)

\put(2.5,55){\circle*{2.7}} \put(12.5,55){\circle*{2.7}}
\put(2.5,65){\circle*{2.7}} \put(12.5,65){\circle*{2.7}}
\put(2.5,55){\line(1,0){10}} \put(2.5,55){\line(0,1){10}}
\put(12.5,65){\line(-1,0){10}} \put(12.5,65){\line(0,-1){10}}
\put(3.5,45){(1)}


\put(22.5,55){\circle{2.7}} \put(32.5,55){\circle{2.7}}
\put(22.5,65){\circle{2.7}} \put(32.5,65){\circle{2.7}}
\put(23.9,55){\line(1,0){7.3}} \put(22.5,56.4){\line(0,1){7.3}}
\put(31.1,65){\line(-1,0){7.3}} \put(32.5,63.6){\line(0,-1){7.3}}
\put(23.5,45){(2)}


\put(42.5,55){\circle{2.7}} \put(52.5,55){\circle{2.7}}
\put(42.5,65){\circle*{2.7}} \put(52.5,65){\circle*{2.7}}
\put(43.9,55){\line(1,0){7.3}} \put(42.5,56.4){\line(0,1){7.3}}
\put(52.5,65){\line(-1,0){10}} \put(52.5,63.6){\line(0,-1){7.3}}
\put(43.5,45){(3)}

\put(62.5,55){\circle*{2.7}} \put(72.5,55){\circle*{2.7}}
\put(62.5,65){\circle*{2.7}} \put(72.5,65){\circle*{2.7}}
\put(62.5,55){\line(1,0){10}} \put(72.5,65){\line(-1,0){10}}
\put(63.5,45){(4)}

\put(82.5,55){\circle{2.7}} \put(92.5,55){\circle{2.7}}
\put(82.5,65){\circle{2.7}} \put(92.5,65){\circle{2.7}}
\put(83.9,55){\line(1,0){7.3}} \put(91.1,65){\line(-1,0){7.3}}
\put(83.5,45){(5)}

\put(102.5,55){\circle*{2.7}} \put(112.5,55){\circle{2.7}}
\put(102.5,65){\circle*{2.7}} \put(112.5,65){\circle{2.7}}
\put(103.9,55){\line(1,0){7.4}} \put(111.1,65){\line(-1,0){7.4}}
\put(103.5,45){(6)}

\put(122.5,55){\circle*{2.7}} \put(132.5,55){\circle*{2.7}}
\put(122.5,65){\circle*{2.7}} \put(132.5,65){\circle*{2.7}}
\put(122.5,55){\line(0,1){10}} \put(132.5,65){\line(0,-1){10}}
\put(132.5,65){\line(-1,0){10}} \put(123.5,45){(7)}

\put(142.5,55){\circle{2.7}} \put(152.5,55){\circle{2.7}}
\put(142.5,65){\circle{2.7}} \put(152.5,65){\circle{2.7}}
\put(142.5,56.4){\line(0,1){7.4}} \put(152.5,63.7){\line(0,-1){7.4}}
\put(151.1,65){\line(-1,0){7.4}} \put(143.5,45){(8)}


\put(2.5,25){\circle*{2.7}} \put(12.5,25){\circle*{2.7}}
\put(2.5,35){\circle{2.7}} \put(12.5,35){\circle{2.7}}
\put(2.5,26.3){\line(0,1){7.3}} \put(11.1,35){\line(-1,0){7.3}}
\put(12.5,33.6){\line(0,-1){7.3}} \put(3.5,14){(9)}

\put(22.5,25){\circle{2.7}} \put(32.5,25){\circle{2.7}}
\put(22.5,35){\circle*{2.7}} \put(32.5,35){\circle*{2.7}}
\put(22.5,26.3){\line(0,1){7.5}} \put(32.5,35){\line(-1,0){10}}
\put(32.5,33.8){\line(0,-1){7.5}} \put(21.5,14){(10)}

\put(42.5,25){\circle*{2.7}} \put(42.5,35){\circle{2.7}}
\put(52.5,30){\circle*{2.7}} \put(62.5,30){\circle*{2.7}}
\put(42.5,25){\line(2,1){10}} \put(52.5,30){\line(-2,1){8.8}}
\put(52.5,30){\line(1,0){10}} \put(47,14){(11)}

\put(72.5,25){\circle{2.7}} \put(72.5,35){\circle*{2.7}}
\put(82.5,30){\circle{2.7}} \put(92.5,30){\circle{2.7}}
\put(73.7,25.6){\line(2,1){7.6}} \put(73.6,34.45){\line(2,-1){7.7}}
\put(83.8,30){\line(1,0){7.3}} \put(77,14){(12)}

\put(107.5,22.5){\circle{2.7}} \put(102.5,30){\circle*{2.7}}
\put(112.5,30){\circle{2.7}} \put(107.5,37.5){\circle*{2.7}}
\put(108.3,23.7){\line(2,3){3.5}} \put(106.7,23.7){\line(-2,3){4.2}}
\put(103.7,30){\line(1,0){7.5}} \put(107.5,37.5){\line(-2,-3){4.3}}
\put(107.5,37.5){\line(2,-3){4.3}} \put(101.5,14){(13)}

\put(122.5,22.5){\circle{2.7}} \put(122.5,30){\circle{2.7}}
\put(132.5,30){\circle*{2.7}} \put(127.5,37.5){\circle{2.7}}
\put(122.5,23.8){\line(0,1){4.9}} \put(123.8,30){\line(1,0){7.4}}
\put(126.7,36.3){\line(-2,-3){3.5}}
\put(128.3,36.3){\line(2,-3){3.6}} \put(122,14){(14)}

\put(142.5,22.5){\circle*{2.7}} \put(142.5,30){\circle*{2.7}}
\put(152.5,30){\circle{2.7}} \put(147.5,37.5){\circle*{2.7}}
\put(142.5,23.8){\line(0,1){5.0}} \put(143.7,30){\line(1,0){7.5}}
\put(147.5,37.5){\line(-2,-3){5}}
\put(148.2,36.45){\line(2,-3){3.6}} \put(142,14){(15)}



\put(2.5,0){\circle{2.7}} \put(9.5,0){\circle{2.7}}
\put(16.5,0){\circle{2.7}} \put(23.5,0){\circle*{2.7}}
\put(30.5,0){\circle{2.7}}
\multiput(3.75,0)(7,0){4}{\line(1,0){4.45}}

\put(10.5,-16){(16)}


\put(40.5,0){\circle*{2.7}} \put(47.5,0){\circle*{2.7}}
\put(54.5,0){\circle*{2.7}} \put(61.5,0){\circle{2.7}}
\put(68.5,0){\circle*{2.7}}
\multiput(41.7,0)(7,0){4}{\line(1,0){4.5}} \put(48.5,-16){(17)}


\put(79.5,-6){\circle*{2.7}} \put(89.5,-6){\circle*{2.7}}
\put(79.5,4){\circle*{2.7}} \put(89.5,4){\circle{2.7}}
\put(84.5,9){\circle{2.7}} \put(79.5,-6){\line(0,1){10}}
\put(88.1,4){\line(-1,0){7.4}} \put(89.5,2.6){\line(0,-1){7.4}}
\put(83.45,7.95){\line(-1,-1){3.1}} \put(85.5,8){\line(1,-1){3.0}}
\put(78.5,-16){(18)}


\put(102.5,-6){\circle{2.7}} \put(112.5,-6){\circle{2.7}}
\put(102.5,4){\circle{2.7}} \put(112.5,4){\circle*{2.7}}
\put(107.5,9){\circle*{2.7}} \put(102.5,-4.6){\line(0,1){7.4}}
\put(111.1,4){\line(-1,0){7.3}} \put(112.5,2.6){\line(0,-1){7.4}}
\put(106.5,8){\line(-1,-1){3.1}} \put(107.5,9){\line(1,-1){5}}

\put(101.5,-16){(19)}


\put(128.5,-6){\circle*{2.7}} \put(148.5,-6){\circle{2.7}}
\put(138.5,9){\circle*{2.7}} \put(138.5,-1){\circle{2.7}}
\put(129.8,-6){\line(1,0){17.4}} \put(129.3,-4.8){\line(2,3){9}}
\put(147.7,-4.8){\line(-2,3){9}} \put(138.5,0.4){\line(0,1){7.2}}
\put(137.3,-1.6){\line(-2,-1){7.6}}
\put(139.7,-1.6){\line(2,-1){7.6}} \put(133,-16){(20)}







\put(2.5,-32.5){\circle{2.7}} \put(12.5,-32.5){\circle*{2.7}}
\put(2.5,-40){\circle{2.7}} \put(12.5,-40){\circle{2.7}}
\put(2.5,-25){\circle*{2.7}} \put(2.5,-38.6){\line(0,1){4.8}}
\put(2.5,-31.1){\line(0,1){4.8}} \put(3.9,-32.5){\line(1,0){7.4}}
\put(12.5,-33.9){\line(0,-1){4.8}} \put(1.5,-51){(21)}


\put(22.5,-32.5){\circle*{2.7}} \put(32.5,-32.5){\circle{2.7}}
\put(22.5,-40){\circle*{2.7}} \put(32.5,-40){\circle*{2.7}}
\put(22.5,-25){\circle{2.7}} \put(22.5,-38.6){\line(0,1){5.3}}
\put(22.5,-31.1){\line(0,1){4.8}} \put(23.9,-32.5){\line(1,0){7.3}}
\put(32.5,-33.9){\line(0,-1){5.2}} \put(21.5,-51){(22)}



\put(44.5,-41){\circle{2.7}} \put(54.5,-41){\circle{2.7}}
\put(44.5,-34){\circle*{2.7}} \put(54.5,-34){\circle*{2.7}}
\put(44.5,-24){\circle{2.7}} \put(54.5,-24){\circle*{2.7}}
\put(45.9,-41){\line(1,0){7.3}} \put(44.5,-32.6){\line(0,1){7.4}}
\put(45.9,-24){\line(1,0){7.3}} \put(54.5,-25.4){\line(0,-1){7.4}}

\put(43.5,-51){(23)}



\put(70,-41){\circle*{2.7}} \put(80,-41){\circle*{2.7}}
\put(70,-34){\circle{2.7}} \put(80,-34){\circle{2.7}}
\put(70,-24){\circle*{2.7}} \put(80,-24){\circle{2.7}}
\put(71.4,-41){\line(1,0){7.3}} \put(70,-32.6){\line(0,1){7.4}}
\put(71.4,-34){\line(1,0){7.3}} \put(80,-25.4){\line(0,-1){7.2}}

\put(69,-51){(24)}

\put(93,-40){\circle*{2.7}} \put(103,-40){\circle*{2.7}}
\put(93,-32.5){\circle*{2.7}} \put(103,-32.5){\circle{2.7}}
\put(93,-25){\circle{2.7}} \put(103,-25){\circle{2.7}}
\multiput(94.3,-40)(0,7.5){3}{\line(1,0){7.4}} \put(91.5,-51){(25)}

\put(113,-40){\circle{2.7}} \put(123,-40){\circle*{2.7}}
\put(113,-32.5){\circle*{2.7}} \put(123,-32.5){\circle{2.7}}
\put(113,-25){\circle*{2.7}} \put(123,-25){\circle{2.7}}
\multiput(114.4,-40)(0,7.5){2}{\line(1,0){7.2}}
\put(113,-38.6){\line(0,1){12.4}}
\multiput(123,-38.6)(0,7.5){2}{\line(0,1){4.7}}
\put(111.5,-51){(26)}


\put(130,-32.5){\circle{2.7}} \put(135,-32.5){\circle{2.7}}
\put(140,-32.5){\circle*{2.7}} \put(145,-32.5){\circle{2.7}}
\put(150,-32.5){\circle*{2.7}} \put(155,-32.5){\circle*{2.7}}
\multiput(131.4,-32.5)(5,0){5}{\line(1,0){2.3}}
\put(135.5,-51){(27)}

\end{picture}
\end{center}
\vspace*{-0.4cm} \caption{Forbidden induced subgraphs of partitioned
black-and-white threshold
graphs}\label{fig:FIS_partitioned_2_threshold_graphs}
\end{figure}
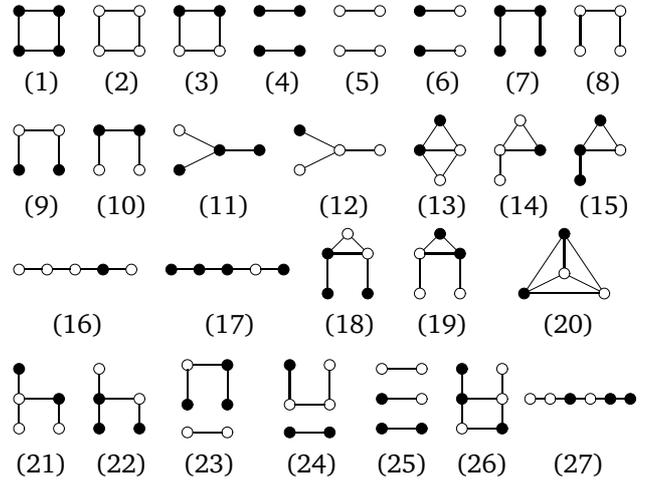

\begin{theorem}
\label{partitioned case}
A partitioned graph is a 2-threshold graph if and only
if it does not contain any of the graphs of
Figure~\vref{fig:FIS_partitioned_2_threshold_graphs}
as an induced subgraph.
\end{theorem}
\begin{proof}
It is easy to see that none of the graphs in
Figure~\ref{fig:FIS_partitioned_2_threshold_graphs} is a 2-threshold
graph. We prove the converse. Let $G$ be a partitioned graph with
none of the graphs of
Figure~\ref{fig:FIS_partitioned_2_threshold_graphs} as an induced
subgraph. It can be checked that none of the graphs in
Figure~%
\ref{fig:FIS_2_threshold_graphs}
is an induced subgraph of the underlying unpartitioned graph.

The case where there are two components with at least two vertices
is easy to check. Assume that $G$ is connected. We prove that there
exists a vertex $x$ such that $N(x)$ or $N[x]$ is exactly the set of
all the white vertices or all the black vertices in $G$. Assume that
$G$ has a house $H=[1;2,3,4,5]$ where $[2,3,4,5]$ is the induced
4-cycle in $H$ and vertex $1$ is the rooftop of $H$ adjacent to
vertices $2$ and $5$. Assume that vertices $1$, $2$, $3$ and $5$ are
colored black and that vertex $4$ is colored white. Let $S \subseteq
V-H$ be the set of vertices that are adjacent to $1$, $2$, $3$ and
$5$, and nonadjacent to~$4$. Add vertex $2$ also to $S$. We claim
that $S$ contains a vertex $x$ such that $N(x)$ or $N[x]$ is the set
of all black vertices in $G$. It can be checked that every black
vertex is in $S$ or is adjacent to a vertex in $S$. The graph $G[S]$
is a threshold graph since $G$ has no octahedron, or gem or a $K_2
\times 2K_2$. The black vertices in $S$ are a clique and the white
vertices are an independent set. It is now easy to check that there
exists a vertex $x \in S$ such that $N(x)$ or $N[x]$ is the set of
black vertices in~$G$.

The other colorings of the house can be dealt with in a similar
fashion.

Assume that $G$ contains no house. Assume that there is a $2K_2$;
$(p,q)+(r,s)$. Assume that $p$ and $q$ are black and that $r$ and
$s$ are white. Consider a shortest path between $\{p,q\}$ and
$\{r,s\}$. Assume that this is a $P_6$;
\[P=[p,q,a,b,r,s].\] Then $a$ is black and $b$ is white.
Assume that $q$ has a white neighbor $\alpha$ and that $r$ has a
black neighbor $\beta$. Then
\[N(\alpha) \cap P = \{p,q,a\} \quad\text{and}\quad
N(\beta)\cap P = \{b,r,s\}.\]
Furthermore, $\alpha$ and $\beta$ are adjacent. Then
$\{q,a,b,\alpha,\beta\}$ induces a house.

Assume that $q$ has only black neighbors. It is easy to check that
there is a vertex $x \in N[q]$ such that $N(x)$ or $N[x]$ is the
set of black vertices.

Other cases can be dealt with in a similar manner.

Assume that there is no house and no $2K_2$. Then $G$ is special. It
is easy to check that the claim holds true in that case. This proves
the theorem. \end{proof}

\section{The switching class of threshold graphs}

If we allow a $\otimes$-operator, which adds a universal vertex,
then we obtain a slightly bigger class of graphs. We call
these graphs extended 2-threshold graphs.
It is easy to see that if $G$ is an extended 2-threshold graph
then so is its complement.

\bigskip
We obtain a class of graphs that is contained
in the 2-threshold graphs if we allow only $\otimes_w$-operators
and $\otimes_b$-operators. We call these graphs
restricted 2-threshold graphs.
It is easy to see that also the restricted 2-threshold graphs
are self-complementary. We prove in this section that
the restricted 2-threshold graphs are exactly the graphs that
are switching equivalent to threshold graphs.

\begin{definition}
Let $G=(V,E)$ be a graph and let $S \subseteq V$. {\em Switching
$G$ with respect to $S$\/} is the operation that adds all edges
to $G$ between nonadjacent pairs with one element in $S$ and the
other not in $S$ and that removes all edges between pairs with
one element in $S$ and the other not in $S$.
\end{definition}

Switching is an equivalence relation on graphs. The equivalence
classes are called switching classes. The work on switching
classes was initiated by van~Lint and
Seidel~\cite{kn:lint,kn:seidel}.

\begin{theorem}
If $G$ is in the switching class of a threshold graph then $G$ is
a restricted 2-threshold graph.
\end{theorem}
\begin{proof}
Assume that $G$ is obtained by switching a threshold graph $H$
with respect to some set $S$ of vertices. Consider the decomposition
tree for the threshold graph $H$. Color the vertices of $S$ white and
color the other vertices black. Change the $\otimes$- and
$\oplus$-operators in the tree as follows.
\begin{enumerate}[\rm 1.]
\item If a vertex $x$ is black and $o(x)=\otimes$ then
change $o(x)$ to $\otimes_b$.
\item If a vertex $x$ is black and $o(x)=\oplus$ then change
$o(x)$  to $\otimes_w$.
\item If a vertex $x$ is white and $o(x)=\otimes$ then change
$o(x)$ to $\otimes_w$.
\item If a vertex $x$ is white and $o(x)=\oplus$ then change
$o(x)$ to $\otimes_b$.
\end{enumerate}
It is easy to see by induction on the height of the decomposition
tree that these operations change the decomposition tree into a
decomposition tree for $G$. \end{proof}

Cameron examined the switching class of cographs~\cite{kn:cameron}.
A graph is a
cograph if every induced subgraph with at least two vertices is
either disconnected or its complement is disconnected~\cite{kn:corneil}.
Thus the class of cographs contains the threshold graphs.
The switching class of cographs are characterized as the graphs
without induced $C_5$, bull, gem, or cogem~\cite{kn:hung}.
In other words, a graph
is switching equivalent to a cograph if and only if it has no
induced subgraph that is switching equivalent to $C_5$.
Thus these graphs are also forbidden for the graphs that are switching
equivalent to a threshold graphs.

\begin{figure}
\setlength{\unitlength}{1.5pt}
\begin{center}
\begin{picture}(110,110)

\put(5,95){\circle*{2}} \put(5,100){\circle*{2}}
\put(5,105){\circle*{2}} \put(15,95){\circle*{2}}
\put(15,100){\circle*{2}} \put(15,105){\circle*{2}}
\put(5,95){\line(1,0){10}} \put(5,100){\line(1,0){10}}
\put(5,105){\line(1,0){10}}


\put(30,95){\circle*{2}} \put(30,105){\circle*{2}}
\put(40,100){\circle*{2}} \put(40,105){\circle*{2}}
\put(50,95){\circle*{2}} \put(50,105){\circle*{2}}
\put(30,95){\line(0,1){10}} \put(50,95){\line(0,1){10}}
\put(30,95){\line(2,1){20}} \put(30,105){\line(2,-1){20}}


\put(65,90){\circle*{2}} \put(65,100){\circle*{2}}
\put(65,110){\circle*{2}} \put(75,90){\circle*{2}}
\put(75,100){\circle*{2}} \put(75,110){\circle*{2}}
\put(65,90){\line(1,0){10}} \put(65,90){\line(0,1){20}}
\put(65,100){\line(1,0){10}} \put(65,110){\line(1,0){10}}
\put(75,90){\line(0,1){20}} \put(65,90){\line(1,1){10}}
\put(65,100){\line(1,1){10}} \put(65,100){\line(1,-1){10}}
\put(65,110){\line(1,-1){10}}

\put(90,90){\circle*{2}} \put(90,100){\circle*{2}}
\put(100,95){\circle*{2}} \put(100,105){\circle*{2}}
\put(110,90){\circle*{2}} \put(110,100){\circle*{2}}
\put(90,90){\line(0,1){10}} \put(110,90){\line(0,1){10}}
\put(100,95){\line(0,1){10}} \put(90,100){\line(2,-1){10}}
\put(90,100){\line(2,1){10}} \put(110,100){\line(-2,1){10}}
\put(110,100){\line(-2,-1){10}}


\put(10,60){\circle*{2}} \put(5,65){\circle*{2}}
\put(15,65){\circle*{2}} \put(5,75){\circle*{2}}
\put(15,75){\circle*{2}} \put(10,80){\circle*{2}}
\put(10,60){\line(-1,1){5}} \put(10,60){\line(1,1){5}}
\put(5,65){\line(0,1){10}} \put(5,65){\line(1,0){10}}
\put(5,65){\line(1,1){10}} \put(10,80){\line(-1,-1){5}}
\put(10,80){\line(1,-1){5}} \put(5,75){\line(1,0){10}}
\put(15,75){\line(0,-1){10}} \put(15,75){\line(-1,-1){10}}
\put(15,65){\line(-1,1){10}}

\put(30,70){\circle*{2}} \put(35,62.5){\circle*{2}}
\put(35,77.5){\circle*{2}} \put(45,62.5){\circle*{2}}
\put(45,77.5){\circle*{2}} \put(50,70){\circle*{2}}
\put(35,62.5){\line(1,0){10}} \put(35,62.5){\line(-2,3){5}}
\put(45,77.5){\line(-1,0){10}} \put(45,77.5){\line(2,-3){5}}
\put(30,70){\line(2,3){5}} \put(45,62.5){\line(2,3){5}}

\put(65,62.5){\circle*{2}} \put(75,62.5){\circle*{2}}
\put(65,72.5){\circle*{2}} \put(75,72.5){\circle*{2}}
\put(65,77.5){\circle*{2}} \put(75,77.5){\circle*{2}}
\put(65,62.5){\line(1,0){10}} \put(65,62.5){\line(0,1){10}}
\put(75,72.5){\line(-1,0){10}} \put(75,72.5){\line(0,-1){10}}

\put(95,60){\circle*{2}} \put(95,70){\circle*{2}}
\put(105,60){\circle*{2}} \put(105,70){\circle*{2}}
\put(100,75){\circle*{2}} \put(100,80){\circle*{2}}
\put(95,60){\line(1,0){10}} \put(95,60){\line(0,1){10}}
\put(95,60){\line(1,3){5}} \put(95,70){\line(1,0){10}}
\put(95,70){\line(1,1){5}} \put(100,75){\line(0,1){5}}
\put(100,75){\line(1,-1){5}} \put(100,75){\line(1,-3){5}}
\put(105,70){\line(0,-1){10}}


\put(2.5,40){\circle*{2}} \put(7.5,40){\circle*{2}}
\put(12.5,40){\circle*{2}} \put(17.5,40){\circle*{2}}
\put(10,30){\circle*{2}} \put(10,50){\circle*{2}}
\put(10,30){\line(-3,4){7.5}} \put(10,30){\line(-1,4){2.5}}
\put(10,30){\line(1,4){2.5}} \put(10,30){\line(3,4){7.5}}
\qbezier(2.5,40)(10,32)(17.5,40) \put(2.5,40){\line(1,0){15}}
\put(10,50){\line(-3,-4){7.5}} \put(10,50){\line(-1,-4){2.5}}
\put(10,50){\line(1,-4){2.5}} \put(10,50){\line(3,-4){7.5}}

\put(35,30){\circle*{2}} \put(35,40){\circle*{2}}
\put(45,30){\circle*{2}} \put(45,40){\circle*{2}}
\put(40,45){\circle*{2}} \put(40,50){\circle*{2}}
\put(35,30){\line(1,0){10}} \put(35,30){\line(0,1){10}}
\put(35,40){\line(1,0){10}} \put(35,40){\line(1,1){5}}
\put(35,40){\line(1,2){5}} \put(45,40){\line(-1,1){5}}
\put(45,40){\line(-1,2){5}} \put(45,40){\line(0,-1){10}}

\multiput(58.75,40)(7.5,0){4}{\circle*{2}}
\put(66.25,32.5){\circle*{2}} \put(73.75,32.5){\circle*{2}}
\put(58.75,40){\line(1,0){22.5}} \put(66.25,32.5){\line(0,1){7.5}}
\put(73.75,32.5){\line(0,1){7.5}}

\put(100,30){\circle*{2}} \put(100,50){\circle*{2}}
\put(95,35){\circle*{2}} \put(95,45){\circle*{2}}
\put(105,35){\circle*{2}} \put(105,45){\circle*{2}}
\put(100,30){\line(-1,1){5}} \put(100,30){\line(1,1){5}}
\put(100,30){\line(0,1){20}} \put(95,35){\line(1,0){10}}
\put(95,35){\line(0,1){10}} \put(95,45){\line(1,0){10}}
\put(95,45){\line(1,1){5}} \put(100,50){\line(1,-1){5}}
\put(105,45){\line(0,-1){10}}

\put(5,0){\circle*{2}} \put(2.5,10){\circle*{2}}
\put(10,17.5){\circle*{2}} \put(17.5,10){\circle*{2}}
\put(15,0){\circle*{2}} \put(5,0){\line(1,0){10}}
\put(5,0){\line(-1,4){2.5}} \put(15,0){\line(1,4){2.5}}
\put(10,17.5){\line(-1,-1){7.5}} \put(10,17.5){\line(1,-1){7.5}}

\put(40,0){\circle*{2}} \put(30,7.5){\circle*{2}}
\put(50,7.5){\circle*{2}} \put(35,15){\circle*{2}}
\put(45,15){\circle*{2}} \put(40,0){\line(-4,3){10}}
\put(40,0){\line(4,3){10}} \put(40,0){\line(-1,3){5}}
\put(40,0){\line(1,3){5}} \put(35,15){\line(-2,-3){5}}
\put(35,15){\line(1,0){10}} \put(45,15){\line(2,-3){5}}

\put(70,0){\circle*{2}} \put(60,7.5){\circle*{2}}
\put(80,7.5){\circle*{2}} \put(65,15){\circle*{2}}
\put(75,15){\circle*{2}} \put(65,15){\line(-2,-3){5}}
\put(65,15){\line(1,0){10}} \put(75,15){\line(2,-3){5}}

\put(100,0){\circle*{2}} \put(92.5,10){\circle*{2}}
\put(92.5,20){\circle*{2}} \put(107.5,10){\circle*{2}}
\put(107.5,20){\circle*{2}} \put(100,0){\line(-3,4){7.5}}
\put(100,0){\line(3,4){7.5}} \put(92.5,10){\line(1,0){15}}
\put(92.5,10){\line(0,1){10}} \put(107.5,10){\line(0,1){10}}

\end{picture}
\end{center}
\caption{Forbidden subgraphs of the switching class of threshold
graphs. Each of them is switching equivalent to $3K_2$, $C_5$ or
$C_4+2K_1$.}\label{fig:FIS_switching_class}
\end{figure}
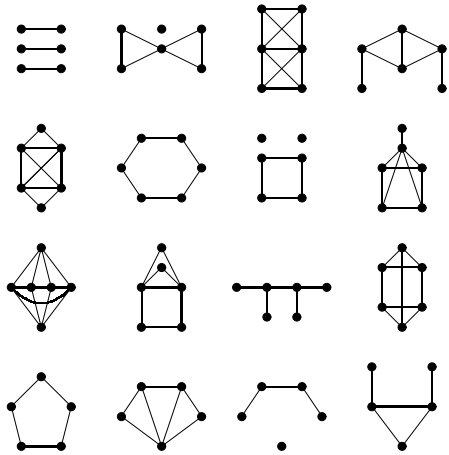

\begin{theorem}
\label{switching threshold}
A graph $G=(V,E)$ is switching equivalent to a threshold graph if and only
if $G$ has no induced subgraph which is switching equivalent to
$3K_2$, $C_5$ or $C_4+2K_1$.
\end{theorem}
\begin{proof}
It is easy to check the necessity. We prove the sufficiency.
Since the graph has no induced subgraph which is switching equivalent
to $C_5$ the graph $G$ is switching equivalent to a cograph $H$.
If $H$ is connected then we can switch it to a disconnected cograph
by switching it with respect to one of the components of the
complement. Henceforth we assume that $H$ is disconnected.
Let $C_1,C_2,\ldots$ be the components of $H$. There are at most
two components with more than one vertex, otherwise the graph has an
induced $3K_2$.
Furthermore, there can be at most one component which
is not a threshold graph already.
First assume that every
component is a threshold graph. Then switch the graph with respect
to a maximal clique in one of them. It is easy to check
that this produces a threshold graph. Assume that $C_1$ does
not induce a
threshold graph.

By induction we can assume that there is a
subset $S$ of $C_1$ such that switching $H[C_1]$ with respect to
$S$ produces a threshold graph $T$. The threshold graph
$T$ has a universal or
isolated vertex $x$. We may assume that $x \not\in S$,
otherwise we can switch $H[C_1]$ with respect to $C_1-S$.
If $x$ is isolated in $T$ then $S=N_H(x)$ and if $x$ is universal
in $T$ then $S=C_1-N_H[x]$.

Assume that $S=N_H(x)$. If $C_1=N_H[x]$ then $H[C_1]$ is a threshold
graph. Assume that $C_1-N_H[x] \neq \es$. Let $\Delta=C_1-N_H[x]$.
For every pair of vertices $a$ and $b$ in $\Delta$ their
neighborhoods in $S$ are ordered by inclusion, otherwise the switch
produces a $C_4$, or $P_4$ or $2K_2$. Since $H[C_1]$ is connected
and has no induced $P_4$ every vertex of $\Delta$ has a neighbor in
$S$. Let $S^{\prime} \subseteq S$ be the subset of vertices in $S$
that have a neighbor in~$\Delta$. Then every pair with one element
in $S^{\prime}$ and one element in $S-S^{\prime}$ is adjacent, since
there is no $P_4$ in $H[C_1]$. The switch makes every vertex of
$S-S^{\prime}$ adjacent to every vertex of $\Delta$. Since the $C_4$
is forbidden at least one of $S-S^{\prime}$ and $S^{\prime}+\Delta$
is a clique.

Assume that there is an
edge $(p,q)$ in $\Delta$ and assume that $p$ has a neighbor
$p^{\prime}$ in $S$ that is not a neighbor of $q$.
Then $[q,p,p^{\prime},x]$ is a $P_4$ in $H$ which is forbidden.
Thus all vertices of every component of $\Delta$ have the
same neighbors in $S$. Call the sets of vertices of $S^{\prime}$ that
have the same neighbors in $\Delta$ the classes of $S^{\prime}$.
Call the unions of those components of $\Delta$ that have the same
neighbors in $S^{\prime}$ the classes of $\Delta$.

Assume that there are two classes $P$ and $Q$ in $S^{\prime}$ and
assume that $P$ is adjacent to a class $P^{\prime}$ of $\Delta$
which is not adjacent to $Q$. Let $Q$ be adjacent to a class
$Q^{\prime}$ of~$\Delta$. Then $P$ is also adjacent to $Q^{\prime}$.
Since $[P^{\prime},P,Q^{\prime},Q]$ has no induced $P_4$, every
vertex of $P$ is adjacent to every vertex of $Q$. Furthermore, since
there is no $C_4$ in the switched graph, at least one of the classes
$P$ and $Q$ is a clique. Thus all of $S^{\prime}$ except possibly
one class is a clique.

Since there is no $2K_2$ in the switched graph
at most one of the components of $\Delta$ has more than one vertex.
Furthermore, if there is a component with more than one vertex
then it has the least number of neighbors in $S^{\prime}$ among all
components of $\Delta$,
otherwise the switch produces
a $2K_2$.

Every pair of nonadjacent classes in $S^{\prime}$ and $\Delta$
become adjacent in the switched graph. Thus at least one of them
is a clique. If there is a class in $S^{\prime}$ that is not a
clique then it is adjacent to all except possibly one class of
$\Delta$ otherwise, there is a $C_4$ in the
switched graph. Consider a component $P$ in $\Delta$ that has
more than one vertex and a class $Q$ in $S^{\prime}$
that is not a clique. If $P$ and $Q$ are not adjacent then
they are adjacent in the switched graph, thus $P$ is a
clique. If $P$ and $Q$ are adjacent then $Q$ is an independent set,
otherwise the switched graph has a $2K_2$.

If $S-S^{\prime}$ is not a clique, then $S^{\prime}+\Delta$
is a clique in the switched graph. Thus $S^{\prime}$ and
$\Delta$ are disjoint cliques in $H_1[C_1]$. This contradicts that
$C_1$ is a component of $H$ and $H[C_1]$ is not a
threshold graph. Thus $S-S^{\prime}$ is empty or else
it is a
clique.

Assume that there is a class $P$ in $\Delta$ which contains an edge.
Let $Q$ be the class in $S^{\prime}$ that is adjacent to~$P$. Then
$Q$ is an independent set, otherwise there is a $2K_2$ in the
switched graph. This class $Q$ is adjacent to all vertices of
$\Delta$, since $P$ has the least neighbors in $S^{\prime}$ among
all classes in $\Delta$. Thus $Q$ is an independent, universal set
of $H[C_1]$.

Assume that there is a component $P$ in $\Delta$ that contains an
edge. Let $Q$ be the neighborhood of $P$ in~$S^{\prime}$. Then
$S-Q=\es$ otherwise, $H$ contains a graph which is switching
equivalent to $3K_2$. Assume that $S-Q=\es$ and that $Q$ is an
independent set. If $Q$ contains more than one vertex then $V-C_1$
is a clique. Switch $H$ with respect to
\[Q+(V-C_1).\] This gives a
threshold graph.

Assume that $\Delta$ is an independent set.
Let $Q$ be a class in $S^{\prime}$ which is not a clique.
Then every vertex of $Q$
is adjacent to all other vertices of $C_1$ except possibly
one vertex in $\Delta$. If there is a vertex in $\Delta$
which is not adjacent to $Q$ then $H$ contains a $C_4+2K_1$
which is a contradiction. Thus $Q$ is adjacent to all vertices
of $C_1-Q$.

First assume that
$Q$ is an independent set.
Switch $H$ with respect to
\[Q+(V-C_1).\] The set $Q$ becomes
a set of isolated vertices. The vertex $x$ is adjacent to the
clique $S-Q$. The graph induced by $\Delta+(S-Q)$ is a threshold
graph. The set $V-C_1$ becomes a clique of vertices
adjacent to all vertices in $C_1-Q$.
Thus $H$ switches to a threshold graph.

Assume that $Q$ is not an independent set.
Switch $H$ with respect to
\[Q+(V-C_1).\]
This gives two disjoint threshold graphs $Q$ and $V-Q$.
As shown at the start, we can switch the graph to a threshold graph.

The case where $S=C_1-N[x]$ is similar. \end{proof}

\begin{theorem}
A graph is switching equivalent to a threshold graph if and only
if it is a restricted 2-threshold graph.
\end{theorem}
\begin{proof}
It is easy to check that none of the forbidden induced subgraphs of
the switching class of threshold graphs is restricted 2-threshold.
\end{proof}

\bibliographystyle{agsm}

\end{document}